\documentclass[reqno]{amsart}
\newtheorem{theorem}{Theorem}
\newtheorem{thm}{Theorem}
\newtheorem*{thm2a}{Corollary 2A}
\newtheorem*{thm2b}{Corollary 2B}

\newtheorem*{thmm}{Theorem}
\newtheorem*{reg}{\textnormal{\textbf{Class}} $\mathbf{\Omega}$}

\newtheorem*{VL}{Vaaler's Lemma}

\newtheorem{cor}{Corollary}
\newtheorem{lemma}[theorem]{Lemma}

\newtheorem{rem}{Remark}
\newtheorem{prop}{Proposition}
\newtheorem*{propp}{Proposition}
\numberwithin{equation}{section}

\usepackage[allcolors=green]
{hyperref}
\usepackage{xcolor}
\usepackage{soul}
\usepackage{setspace}
\usepackage{newtxmath}
\usepackage{mathrsfs}
\usepackage{accents}
\usepackage{bigints}
\usepackage{dsfont}

\usepackage{scalerel}
\begin{document}

\title[Lattice points in shrinking Cygan-Kor{\'a}nyi spherical shells]{Lattice point counting statistics for 3-dimensional shrinking Cygan-Kor{\'a}nyi spherical shells}
\author{YOAV A. GATH}
\address{Centre for Mathematical Sciences, Wilberforce Road, Cambridge CB3 0WA, United Kingdom}
\curraddr{8 Cosin Court, Cambridge CB2 1QU, United Kingdom}
\email{yg395@cantab.ac.uk}
%\address{Department of Mathematics, The University of Hong Kong, Pokfulam Road, Pokfulam, Hong Kong}
%\email{yklau@maths.hku.hk}
\thanks{Yoav A. Gath has received funding from the European Research Council (ERC) under the European Union’s Horizon 2020 research and innovation programme
(grant agreement No. 803711)}


\begin{abstract}
Let $\mathcal{E}(x;\omega)$ be the error term for the number of integer lattice points lying inside a $3$-dimensional Cygan-Kor{\'a}nyi spherical shell of inner radius $x$ and gap width $\omega(x)>0$. Assuming that $\omega(x)\to0$ as $x\to\infty$, and that $\omega$ satisfies suitable regularity conditions, we prove that $\mathcal{E}(x;\omega)$, properly normalized, has a limiting distribution. Moreover, we show that the corresponding distribution is moment-determinate, and we give a closed form expression for its moments. As a corollary, we deduce that the limiting distribution is the standard Gau$\mathbf{\ss}$ian measure whenever $\omega$ is slowly varying. We also construct gap width functions $\omega$, whose corresponding error term has a limiting distribution that is absolutely continuous with a non-Gau$\mathbf{\ss}$ian density.
%if $\omega(x)\to0$ as $x\to\infty$, the error term (properly normalized)  We prove a general structure theorem for probability measures that arise as limiting distributions, as $X\to\infty$, of the random variable $x\mapsto\widehat{\mathcal{E}}(x;\omega)/\sigma(X;\omega)$, where $x$ varies uniformly in the dyadic segment $X<x<2X$, and $\omega(x)\to0$ as $x\to\infty$ is assumed to satisfy suitable regularity conditions. As a consequence, we deduce that the limiting distribution is the standard Gau$\mathbf{\ss\ss}$ian measure whenever $\omega$ is slowly varying. For non-slowly varying functions (of suitable type) we show that the limiting distribution is absolutely continuous, and moreover, we give an explicit construction of the corresponding density in terms of the gap width function $\omega$; in particular, we show that the limiting distribution is no longer Gau$\mathbf{\ss\ss}$ian in this case.  %We show that any distribution arising this way must be absolutely continuous, with a smooth symmetric density, and is uniquely determined by its moments.
\end{abstract}

%\begin{abstract}
%In this study, we investigate the behavior of the random variable $x\mapsto\widehat{\mathcal{E}}(x;\omega)/\sigma(X;\omega)$ as $X$ tends to infinity. Here, $N(x;\omega)$ represents the number of integer lattice points inside a three-dimensional Cygan-Kor{'a}nyi spherical shell, characterized by an inner radius $x$ and a positive gap width $\omega(x)$. The corresponding error term is denoted as $\mathcal{E}(x;\omega)$.

%We establish a general structure theorem for probability measures that emerge as limiting distributions of the aforementioned random variable, where $x$ uniformly varies within the dyadic segment $X<x<2X$. The function $\omega(x)$ approaches zero as $x$ tends to infinity, and it satisfies certain regularity conditions. Consequently, we demonstrate that the limiting distribution follows the standard Gaussian measure when $\omega$ exhibits slow variation.

%Moreover, our analysis reveals that any distribution obtained in this manner must be absolutely continuous, possessing a smooth and symmetric density. Furthermore, the distribution is uniquely determined by its moments.
%\end{abstract}

\maketitle
\tableofcontents
\section{Introduction, notation and statement of results}
\subsection{Introduction}
Let $N(x)$ be the counting function for the number of integer lattice points lying inside a $3$-dimensional Cygan-Kor{\'a}nyi ball of radius $x$,
\begin{equation*}
N(x)=\big|\mathbb{Z}^{3}\cap\delta_{x}\mathcal{B}\big|\,,
\end{equation*}
where $\mathcal{B}=\big\{\mathbf{u}\in\mathbb{R}^{3}:|\mathbf{u}|_{{\scriptscriptstyle Cyg}}\leq1\big\}$ is the unit ball with respect to the Cygan-Kor{\'a}nyi norm
\begin{equation*}
|\mathbf{u}|_{{\scriptscriptstyle Cyg}}=\Big(\Big(u^{2}_{1}+u^{2}_{2}\Big)^{2}+u^{2}_{3}\Big)^{1/4}\,,
\end{equation*}
and $\delta_{x}\mathbf{u}=(xu_{{\scriptscriptstyle 1}}, xu_{{\scriptscriptstyle2}},x^{2}u_{{\scriptscriptstyle3}})$ with $x>0$ are the Heisenberg dilations. It is clear that $N(x)$ grows for large $x$ like $\textit{vol}(\mathcal{B})x^{4}$, where $\textit{vol}(\cdot)$ is the standard Euclidean
volume. Define the error term $\mathcal{E}(x)$ to be
\begin{equation*}
\mathcal{E}(x)=N(x)-\textit{vol}(\mathcal{B})x^{4}\,,
\end{equation*}
We remark that estimating $N(x)$ (resp. bounding $\mathcal{E}(x)$) is a natural problem that arises in relation to the homogeneous structure imposed on the first Heisenberg group $\mathbb{H}_{1}$ when realized over $\mathbb{R}^{3}$ (see \cite{garg2015lattice,gath2019analogue}).\\
In \cite{gath2017best} aspects relating to the order of magnitude of the error term have been investigated, the main result being the determination of the exponent of growth $\kappa:=\sup\{\alpha>0:|\mathcal{E}(x)|\ll x^{4-\alpha}\}=2$. In \cite{3Dgathdistribution} we turned attention to more subtle questions regarding the nature in which $N(x)$ fluctuates about its expected value $\textit{vol}(\mathcal{B})x^{4}$. In order for us to place the results of the present paper in a proper context, we quote the following theorem.   
\begin{thmm}[\cite{3Dgathdistribution}]
The suitably normalized error term $\widehat{\mathcal{E}}(x)=\mathcal{E}(x)/x^{2}$ has a limiting distribution, namely, there exists a probability measure $\nu$ such that for any bounded (piecewise)-continuous function $\mathcal{F}$,
\begin{equation*}
(\star)\quad\quad\lim\limits_{X\to\infty}\frac{1}{X}\int\limits_{ X}^{2X}\mathcal{F}\big(\widehat{\mathcal{E}}(x)\big)\textit{d}x=\int\limits_{-\infty}^{\infty}\mathcal{F}(\alpha)\textit{d}\nu(\alpha)\,.
\end{equation*}
Moreover, the following holds.
\begin{itemize}
\item[\textnormal{\textbf{(a)}}] The limiting distribution is absolutely continuous $\textit{d}\nu(\alpha)=\mathcal{P}(\alpha)\textit{d}\alpha$.
\item[\textnormal{\textbf{(b)}}] The corresponding density $\mathcal{P}$ is real analytic, and in particular has support on all of the real line. It satisfies for any non-negative integer $j\geq0$ and any $\alpha\in\mathbb{R}$, $|\alpha|$ sufficiently large in terms of $j$, the decay estimate $|\mathcal{P}^{(j)}(\alpha)|\leq\exp{\Big(-\rho|\alpha|\exp{\big(\rho|\alpha|\big)}\Big)}$, where $\rho>0$ is an absolute constant.
\item[\textnormal{\textbf{(c)}}] $\int_{-\infty}^{\infty}\alpha\mathcal{P}(\alpha)\textit{d}\alpha=0$, while $\int_{-\infty}^{\infty}\alpha^{k}\mathcal{P}(\alpha)\textit{d}\alpha<0$ for all odd integers $k>1$. In particular, the limiting distribution is asymmetric.
\item[\textnormal{\textbf{(d)}}] $(\star)$ extends to the class of (piecewise)-continuous functions $\mathcal{F}$ of polynomial growth. In particular, $\widehat{\mathcal{E}}(x)$ has finite moments of any order.
\end{itemize}
%
%
%The limiting distribution is absolutely continuous $\textit{d}\nu_{{\scriptscriptstyle1}}(\alpha)=\mathcal{P}_{{\scriptscriptstyle1}}(\alpha)\textit{d}\alpha$, with the corresponding density being a real analytic function that satisfies for any non-negative integer $j\geq0$ and any $\alpha\in\mathbb{R}$, $|\alpha|$ sufficiently large in terms of $j$, the decay estimate

%that satisfies for any non-negative integer $j$, and any $\alpha\in\mathbb{R}$, $|\alpha|$ sufficiently large in terms of $j$ and $q$, the decay estimate
%
%
%\begin{equation*}
%\big|\mathcal{P}^{(j)}_{{\scriptscriptstyle q}}(\alpha)\big|\leq\left\{
%        \begin{array}{ll}
%            \exp{\bigg(-\frac{\pi}{2}|\alpha|\exp{\big(\rho|\alpha|\big)}\bigg)} & ;\,q=1 
%            \\\\
%            \exp{\Big(-|\alpha|^{4-\beta/\log\log{|\alpha|}}\Big)} & ;\, q\geq3\,,
%        \end{array}
%    \right.
%\end{equation*}
%
%
%where $\rho,\,\beta>0$ are absolute constants. Moreover, the weak convergence can be strengthened considerably, in the sense that $(\star)$ extends to the class of (piecewise)-continuous functions of polynomial growth in the case $q=1$, and of quadratic growth in the case $q\geq3$. 
%In addition, the density $\mathcal{P}_{q}(\alpha)$ can be extended to the whole complex plane $\mathbb{C}$ as an entire function of $\alpha$ in the case $q\geq3$.   
\end{thmm}
\noindent
Motivated by the above quoted theorem, we ask what can be said about the statistics of the (suitably normalized) error term corresponding to the number of integer lattice points lying inside a $3$-dimensional Cygan-Kor{\'a}nyi spherical shell $\delta_{(x+\omega(x))}\mathcal{B}\setminus\delta_{x}\mathcal{B}$ of inner radius $x$ and gap width $\omega(x)>0$.\\
There are two extreme regimes that one can consider: one in which the width of the spherical shell grows as the the radius increases, and the other in which it shrinks. The so-called ``saturation" regime in which the width stays bounded should be thought of as somewhat of a transition phase between the two.\\
In the present paper we shall focus our attention on the case where the width shrinks. Our goal will be to state under which conditions a limiting distribution exists, and describe the properties of this limiting distribution in relation to the gap width function $\omega$. In particular, we shall aim, whenever possible, to establish analogues of items \textnormal{(a)-(d)} in the above quoted theorem. As we shall see, the picture that emerges will be quite different. Let us now proceed to introduce to the setup we shall be working with.      
\subsection{Lattice points in shrinking $3$-dimensional Cygan-Kor{\'a}nyi spherical shells}
We begin by introducing the relevant notation and definitions we sall be using throughout the paper. The counting function $N(x;\omega)$ for the number of integer lattice points lying inside a $3$-dimensional Cygan-Kor{\'a}nyi spherical shell $\delta_{(x+\omega(x))}\mathcal{B}\setminus\delta_{x}\mathcal{B}$ of inner radius $x$ and gap width $\omega(x)>0$ is defined by  
\begin{equation*}
\begin{split}
N(x;\omega)&= \big|\mathbb{Z}^{3}\cap\big(\delta_{(x+\omega(x))}\mathcal{B}\setminus\delta_{x}\mathcal{B}\big)\big|\\
&=N(x+\omega(x))-N(x)\,.
\end{split}
\end{equation*}
We shall be exclusively concerned with the case where the width of the spherical shell shrinks as the radius increases, that is $\omega(x)\to0$ as $x\to\infty$. Note that we are not requiring that $\omega$ be monotonic. We shall however require that it satisfy suitable regularity conditions. We shall specify all of these requirements shortly.\\
The expected value of $N(x;\omega)$ is the volume of the corresponding spherical shell, that is $\textit{vol}\big(\delta_{(x+\omega(x))}\mathcal{B}\setminus\delta_{x}\mathcal{B}\big)=\textit{vol}(\mathcal{B})\sum_{j=1}^{4}\binom{4}{j}x^{4-j}\omega^{j}(x)$. Define the error term $\mathcal{E}(x;\omega)$ to be
\begin{equation*}
\mathcal{E}(x;\omega)=N(x;\omega)-\textit{vol}(\mathcal{B})\sum_{j=1}^{4}\binom{4}{j}x^{4-j}\omega^{j}(x)\,,
\end{equation*}
and set 
\begin{equation*}
\widehat{\mathcal{E}}(x;\omega)=\mathcal{E}(x;\omega)/x^{2}\,.
\end{equation*}
For $X>0$, let $\sigma^{2}(X;\omega)$ be the variance of $\widehat{\mathcal{E}}(x;\omega)$ defined by
\begin{equation*}
\sigma^{2}(X;\omega)=\frac{1}{X}\int\limits_{X}^{2X}\widehat{\mathcal{E}}^{2}(x;\omega)\textit{d}x\,.
\end{equation*}
Consider the random variable $x\mapsto\widehat{\mathcal{E}}(x;\omega)/\sigma(X;\omega)$ with $x$ uniformly distributed in the segment $X<x<2X$, and let $\mu_{{\scriptscriptstyle X,\omega}}$ be the corresponding distribution, defined for an interval $\mathcal{A}$ on the real line by
\begin{equation*}
\int_{\mathcal{A}}\textit{d}\mu_{{\scriptscriptstyle X,\omega}}(\alpha)=\frac{1}{X}\textit{meas}\big\{X<x<2X:\widehat{\mathcal{E}}(x;\omega)/\sigma(X;\omega)\in\mathcal{A}\big\}\,,
\end{equation*}
where $\textit{meas}$ is the Lebesgue measure.\\
We seek to establish the existence of a limiting distribution for $\widehat{\mathcal{E}}(x;\omega)/\sigma(X;\omega)$, that is, prove that there exists a probability measure $\mu_{{\scriptscriptstyle\omega}}$ such that $\mu_{{\scriptscriptstyle X,\omega}}\to\mu_{{\scriptscriptstyle\omega}}$ weakly, as $X\to\infty$. To achieve this, we must have some control over the gap width function $\omega$. We introduce the following class.
\begin{reg}
Let $\omega\in C^{{\scriptscriptstyle2}}(\mathbb{R}_{>0})$, and for $X>0$ let $\mathcal{U}^{\omega}_{{\scriptscriptstyle X}}=\{X\leq x\leq2X:\omega^{{\scriptscriptstyle(1)}}(x)=0\}$ and $\mathcal{V}^{\omega}_{{\scriptscriptstyle X}}:=\{X\leq x\leq2X:\omega^{{\scriptscriptstyle(2)}}(x)=0\}$. We shall say that $\omega\in\mathbf{\Omega}$ if it satisfies the regularity conditions listed below.
\begin{itemize}
\item[\textbf{(1)}] $\omega$ is strictly positive $\omega(x)>0$, and vanishes at infinity $\omega(x)\to0$ as $x\to\infty$.
\item[\textbf{(2)}] $|\omega^{{\scriptscriptstyle(1)}}(x)|<\frac{1}{2}$ for all sufficiently large $x>0$.
\item[\textbf{(3)}] $\mathcal{U}^{\omega}_{{\scriptscriptstyle X}}$ and $\mathcal{V}^{\omega}_{{\scriptscriptstyle X}}$ are finite for all sufficiently large $X$, and the following bounds hold:
\begin{itemize}
\item[\textbf{(3\textnormal{\textbf{a}})}] $\big|\mathcal{U}^{\omega}_{{\scriptscriptstyle X}}\big|\max\limits_{X\leq\,x\,\leq2X}\omega(x)\ll X^{1/2}$.
\item[\textbf{(3\textnormal{\textbf{b}})}] There exists $\xi_{{\scriptscriptstyle\omega}}>0$ such that $\big|\mathcal{V}^{\omega}_{{\scriptscriptstyle X}}\big|\ll X^{1-\xi_{{\scriptscriptstyle\omega}}}$.
\end{itemize}
\item[\textbf{(4)}] For a non-negative integer $j$ and $X>0$, let
\begin{equation*}
\mathscr{M}_{j}(X;\omega)=\frac{1}{X}\int\limits_{X}^{2X}\big(\omega(x)\log{\omega(x)}\big)^{j}\textit{d}x\,.
\end{equation*}
Then the following holds:
\begin{itemize}
\item[\textbf{(4\textnormal{\textbf{a}})}] $\tau(\omega):=\inf\{\alpha>0: \mathscr{M}_{2}(X;\omega)\gg X^{-\alpha}\}=0$.
\item[\textbf{(4\textnormal{\textbf{b}})}] The limit
\begin{equation*}
\mathscr{L}_{j}(\omega):=\lim_{X\to\infty}\frac{\mathscr{M}_{j}(X;\omega)}{\mathscr{M}^{j/2}_{2}(X;\omega)}
\end{equation*}
exists for all $j\equiv0\,(2)$.
\item[\textbf{(4\textnormal{\textbf{c}})}] The sequence $\mathfrak{m}_{\omega,j}=\frac{j!}{2^{j/2}(j/2)!}\mathscr{L}_{j}(\omega)$ with $j\equiv0\,(2)$ satisfies Carleman’s criterion, that is, we have
\begin{equation*}
\underset{j\equiv0\,(2)}{\sum_{j=0}^{\infty}}\mathfrak{m}^{-1/j}_{\omega,j}=\infty\,.
\end{equation*}
\end{itemize}
\end{itemize}
\end{reg}
%
%and explain their role in the various stages of proof
\noindent
We shall comment on the above conditions later on, and explain their role in the various stages of proofs (see $\S1.3$, proof outline). Let us now proceed to state our main findings.
\subsection{Statement of the main results}
\begin{thm}[Main Theorem]\label{Main Theorem} Let $\omega\in\mathbf{\Omega}$. Then $\widehat{\mathcal{E}}(x;\omega)/\sigma(X;\omega)$ has a limiting distribution, namely, there exists a probability measure $\mu_{{\scriptscriptstyle\omega}}$ such that for any bounded (piecewise)-continuous function $\mathcal{F}$ whose points of discontinuity are points of continuity of $\mu_{{\scriptscriptstyle\omega}}$, we have 
\begin{equation}\label{eq:1.1}
\lim\limits_{X\to\infty}\frac{1}{X}\int\limits_{ X}^{2X}\mathcal{F}\bigg(\frac{\widehat{\mathcal{E}}(x;\omega)}{\sigma(X;\omega)}\bigg)\textit{d}x=\int\limits_{-\infty}^{\infty}\mathcal{F}(\alpha)\textit{d}\mu_{{\scriptscriptstyle\omega}}(\alpha)\,.
\end{equation}
Moreover, the following holds.
%$\textit{d}\mu_{{\scriptscriptstyle X,\omega}}\rightarrow\textit{d}\mu_{{\scriptscriptstyle\omega}}$ weakly, as $X\to\infty$. Moreover, the following holds.
%the gap width function $h$ satisfy our standing assumptions, and suppose that $\tau(h)=0$.  Then a necessary and sufficient condition for the distributions $\textit{d}\mu_{{\scriptscriptstyle h,X}}$ to converge weakly, as $X\to\infty$, is for the limit
%
%
%\begin{equation*}
%\mathscr{L}_{k}(\Check{\omega}):=\lim_{X\to\infty}\mathscr{L}_{k}(X;\Check{\omega})/\mathscr{L}^{k/2}_{2}(X;\Check{\omega})
%\end{equation*}
%
%
%
%
\begin{itemize}
%\item[\textnormal{\textbf{(A)}}] The limiting distribution $\textit{d}\mu_{{\scriptscriptstyle h}}$ is absolutely continuous $\textit{d}\mu_{{\scriptscriptstyle h}}(\alpha)=\mathcal{P}_{{\scriptscriptstyle h}}(\alpha)\textit{d}\alpha$.
%\item[\textnormal{\textbf{(B)}}] The corresponding density $\mathcal{P}_{{\scriptscriptstyle h}}$ extends to an entire function on the complex plane, and in particular has support on all of the real line.
%It satisfies for any non-negative integer $j\geq0$ and any $\alpha\in\mathbb{R}$, $|\alpha|$ sufficiently large in terms of $j$, the decay estimate $|\mathcal{P}^{(j)}_{{\scriptscriptstyle1}}(\alpha)|\leq?$.
\item[\textnormal{\textbf{(A)}}] $\mu_{{\scriptscriptstyle\omega}}$ is moment-determinate, i.e., it is uniquely characterised by its moments, and these are given by
%has finite moments of any order $\int_{-\infty}^{\infty}|\alpha|^{k}\textit{d}\mu_{{\scriptscriptstyle h}}(\alpha)<\infty$, $k=1,2,\ldots$. and is moment-determinatethe unique probability measure satisfying
\begin{equation}\label{eq:1.2}
\int\limits_{-\infty}^{\infty}\alpha^{j}\textit{d}\mu_{{\scriptscriptstyle\omega}}(\alpha)=\left\{
        \begin{array}{ll}
           \frac{j!}{2^{j/2}(j/2)!}\mathscr{L}_{j}(\omega) & ;\,j\equiv0\,(2)
            \\\\
           0 & ;\, j\equiv1\,(2)\,.
        \end{array}
    \right.
\end{equation}
In particular, if the limiting distribution has a density, then this density is symmetric.
\item[\textnormal{\textbf{(B)}}] \eqref{eq:1.1} extends to the class of (piecewise)-continuous functions $\mathcal{F}$ of polynomial growth whose points of discontinuity are points of continuity of $\mu_{{\scriptscriptstyle\omega}}$. In particular, $\widehat{\mathcal{E}}(x;\omega)/\sigma(X;\omega)$ has finite moments of any order.\\
%The weak convergence can be strengthened considerably, in the sense that holds for any (piecewise)-continuous function $\mathcal{F}$ of polynomial growth. 
\end{itemize}
\end{thm}
%
%
%\text{\calligra{\bfseries{Example: The standard Gaussian measure.}}}
%
%
%
%
\begin{rem}
Let $\omega\in\mathbf{\Omega}$. Then $\mathscr{L}_{2}(\omega)=1$ by definition, while by H$\ddot{\text{o}}$lder's inequality $\mathscr{L}_{j}(\omega)\geq1$ for even integers $j>2$, hence $\int_{-\infty}^{\infty}\alpha^{2}\textit{d}\mu_{{\scriptscriptstyle\omega}}(\alpha)=1$ and $\int_{-\infty}^{\infty}\alpha^{j}\textit{d}\mu_{{\scriptscriptstyle\omega}}(\alpha)\geq\frac{j!}{2^{j/2}(j/2)!}$ respectively. The extremal case where equality holds for all even integers $j$ is considered below.
%It follows that the moments of $\mu_{{\scriptscriptstyle\omega}}$ grow at least as fast as those of the standard Gau$\mathbf{\ss\ss}$ian.
\end{rem}
\subsubsection*{\textbf{The case of slowly varying functions.}} Let $\omega\in\mathbf{\Omega}$ be a slowly varying function, i.e., $\omega$ satisfies for any $\lambda>0$ the relation $\lim_{x\to\infty}\omega(\lambda x)/\omega(x)=1$. We remark that the convergence is uniform over $\lambda$ in compact sets. Relevant examples of slowly varying functions include $\frac{1}{\log{\log{x}}}, \frac{1}{\log{x}}, \frac{1}{\exp{(\sqrt{\log{x}})}}$. Now, let $j$ be a positive integer and observe the following. Given $X>0$, by the mean value theorem for integrals we may find $X\leq x_{{\scriptscriptstyle j}}, x_{{\scriptscriptstyle2}}\leq2X$ such that $\big(\omega(x_{{\scriptscriptstyle j}})\log{\omega(x_{{\scriptscriptstyle j}})}\big)^{j}=\mathscr{M}_{j}(X;\omega)$ and $\big(\omega(x_{{\scriptscriptstyle2}})\log{\omega(x_{{\scriptscriptstyle2}})}\big)^{2}=\mathscr{M}_{2}(X;\omega)$. Note that $\omega(x)\log{\omega(x)}$ is also a slowly varying function since $\omega$ vanishes at infinity. It follows that
\begin{equation*}
\frac{\mathscr{M}_{j}(X;\omega)}{\mathscr{M}^{j/2}_{2}(X;\omega)}=\bigg(\frac{\omega(x_{{\scriptscriptstyle j}})\log{\omega(x_{{\scriptscriptstyle j}})}}{\omega(x_{{\scriptscriptstyle2}})\log{\omega(x_{{\scriptscriptstyle2}})}}\bigg)^{j}\underset{X\to\infty}{\longrightarrow}1\,.
\end{equation*}
In particular, we find that $\mathscr{L}_{j}(\omega)=1$ for all even integers $j$. The RHS of \eqref{eq:1.2} in this case corresponds to the moments of the standard Gau$\mathbf{\ss}$ian measure. According to item \textnormal{(A)} $\mu_{{\scriptscriptstyle\omega}}$ is moment-determinate, and so we conclude that $\textit{d}\mu_{{\scriptscriptstyle\omega}}(\alpha)=\frac{1}{\sqrt{2\pi}}\exp{(-\alpha^{2}/2)}\textit{d}\alpha$. Specializing $\mathcal{F}$ in \eqref{eq:1.1} to be the indicator function of a ray, we deduce the following. 
\begin{cor}
For any slowly varying function $\omega\in\mathbf{\Omega}$ and any $t\in\mathbb{R}$, we have
\begin{equation}\label{eq:1.3}
\lim\limits_{X\to\infty}\frac{1}{X}\textit{meas}\Bigg\{X<x<2X:\frac{\widehat{\mathcal{E}}(x;\omega)}{\sigma(X;\omega)}<t\Bigg\}=\frac{1}{\sqrt{2\pi}}\int\limits_{-\infty}^{t}\exp{\bigg(-\frac{\alpha^{2}}{2}\bigg)}\textit{d}\alpha\,.
\end{equation}
\end{cor}
\begin{rem}
Hughes and Rudnick \cite{HughesRudnick} have established an analogue of the above corollary for the error term corresponding to the number of integer lattice points in shrinking Euclidean annuli (see also \cite{Wigman1},\cite{Wigman2} for further results).
\end{rem}
\noindent
Our next example concerns the limiting distribution $\mu_{{\scriptscriptstyle\omega}}$ of $\widehat{\mathcal{E}}(x;\omega)/\sigma(X;\omega)$ for non-slowly varying functions $\omega\in\mathbf{\Omega}$ of suitable type.
\subsubsection*{\textbf{Product of a slowly varying function and almost periodic polynomials.}} Let $\varphi_{{\scriptscriptstyle\ell}}(t)=|p_{{\scriptscriptstyle\ell}}(\exp{(2\pi it)})|^{2}$, and $\lambda_{{\scriptscriptstyle\ell}}\in\mathbb{R}$ with $\ell=1,\ldots,n$ be given, where $p_{{\scriptscriptstyle\ell}}$ are polynomials with no roots on the unit circle. Fix an integer $A>1$, and let $\omega_{{\scriptscriptstyle\times}}(x):=\big\{\prod_{\ell=1}^{n}\varphi_{{\scriptscriptstyle\ell}}\big(\lambda_{{\scriptscriptstyle\ell}}(\log{x})^{A}\big)\big\}(\log{x})^{-A}$ and $\omega_{{\scriptscriptstyle+}}(x):=\big\{\sum_{\ell=1}^{n}\varphi_{{\scriptscriptstyle\ell}}\big(\lambda_{{\scriptscriptstyle\ell}}(\log{x})^{A}\big)\big\}(\log{x})^{-A}$. We will show ($\S3$, Lemma $7$) that $\omega_{{\scriptscriptstyle\times}}, \omega_{{\scriptscriptstyle+}}\in\mathbf{\Omega}$. Moreover, denoting by $\Phi_{{\scriptscriptstyle\times}}(\mathbf{t})=\Phi_{{\scriptscriptstyle\times}}((t_{{\scriptscriptstyle1}},\ldots, t_{{\scriptscriptstyle n}}))=\prod_{\ell=1}^{n}\varphi_{\ell}(t_{{\scriptscriptstyle\ell}})$ and $\Theta_{{\scriptscriptstyle+}}(\mathbf{t})=\Theta_{{\scriptscriptstyle+}}((t_{{\scriptscriptstyle1}},\ldots, t_{{\scriptscriptstyle n}}))=\sum_{\ell=1}^{n}\varphi_{\ell}(t_{{\scriptscriptstyle\ell}})$, it will be shown that if $\lambda_{{\scriptscriptstyle1}},\ldots,\lambda_{{\scriptscriptstyle n}}$ are linearly independent over $\mathbb{Z}$, then for even integers $j$ 
\begin{equation*}
\begin{split}
&(\dag)\quad\lim_{X\to\infty}\frac{\mathscr{M}_{j}(X;\omega_{{\scriptscriptstyle\times}})}{\mathscr{M}^{j/2}_{2}(X;\omega_{{\scriptscriptstyle\times}})}=\bigg(\frac{\Vert\Phi_{{\scriptscriptstyle\times}}\Vert_{{\scriptscriptstyle j}}}{\Vert\Phi_{{\scriptscriptstyle\times}}\Vert_{{\scriptscriptstyle2}}}\bigg)^{j}\\
&(\ddag)\quad\lim_{X\to\infty}\frac{\mathscr{M}_{j}(X;\omega_{{\scriptscriptstyle+}})}{\mathscr{M}^{j/2}_{2}(X;\omega_{{\scriptscriptstyle+}})}=\bigg(\frac{\Vert\Theta_{{\scriptscriptstyle+}}\Vert_{{\scriptscriptstyle j}}}{\Vert\Theta_{{\scriptscriptstyle+}}\Vert_{{\scriptscriptstyle2}}}\bigg)^{j}\,,
\end{split}
\end{equation*}
where $\Vert\Phi_{{\scriptscriptstyle\times}}\Vert_{{\scriptscriptstyle j}}=\Big(\int_{\mathbf{t}\in[0,1)^{n}}\Phi^{j}_{{\scriptscriptstyle\times}}(\mathbf{t})\textit{d}\mathbf{t}\Big)^{1/j}$ and $\Vert\Theta_{{\scriptscriptstyle+}}\Vert_{{\scriptscriptstyle j}}=\Big(\int_{\mathbf{t}\in[0,1)^{n}}\Theta^{j}_{{\scriptscriptstyle+}}(\mathbf{t})\textit{d}\mathbf{t}\Big)^{1/j}$ denote the $j$-norm of $\Phi_{{\scriptscriptstyle\times}}$ and $\Theta_{{\scriptscriptstyle+}}$ respectively. From the above, we deduce the following. 
\begin{thm2a}[\textnormal{Multiplicative version}]
Let $\omega_{{\scriptscriptstyle\times}}$ be defined as above. Then the normalized error term $\widehat{\mathcal{E}}(x;\omega_{{\scriptscriptstyle\times}})/\sigma(X;\omega_{{\scriptscriptstyle\times}})$ has a limiting distribution $\mu_{\omega_{{\scriptscriptstyle\times}}}$ as detailed in Theorem $1$. Moreover, if $\lambda_{{\scriptscriptstyle1}},\ldots,\lambda_{{\scriptscriptstyle n}}$ are linearly independent over $\mathbb{Z}$, then $\mu_{\omega_{{\scriptscriptstyle\times}}}$ is absolutely continuous  $\textit{d}\mu_{\omega_{{\scriptscriptstyle\times}}}(\alpha)=\mathcal{P}_{\Phi_{{\scriptscriptstyle\times}}}(\alpha)\textit{d}\alpha$, where the density $\mathcal{P}_{\Phi_{{\scriptscriptstyle\times}}}$ depends only on $\Phi_{{\scriptscriptstyle\times}}$ and is otherwise independent of the $\lambda_{{\scriptscriptstyle\ell}}^{'}\text{s}$ and the exponent $A$, and is given by
\begin{equation}\label{eq:1.4}
\mathcal{P}_{\Phi_{{\scriptscriptstyle\times}}}(\alpha)=\frac{\Vert\Phi_{{\scriptscriptstyle\times}}\Vert_{{\scriptscriptstyle2}}}{\sqrt{2\pi}}\int\limits_{\mathbf{t}\in[0,1)^{n}}\frac{1}{\Phi_{{\scriptscriptstyle\times}}(\mathbf{t})}\exp{\bigg(-\frac{1}{2}\frac{\Vert\Phi_{{\scriptscriptstyle\times}}\Vert^{{\scriptscriptstyle2}}_{{\scriptscriptstyle2}}}{\Phi^{{\scriptscriptstyle2}}_{{\scriptscriptstyle\times}}(\mathbf{t})}\alpha^{2}\bigg)}\textit{d}\mathbf{t}\,.
\end{equation}
\end{thm2a}
\begin{thm2b}[\textnormal{Additive version}]
Let $\omega_{{\scriptscriptstyle+}}$ be defined as above. Then the normalized error term $\widehat{\mathcal{E}}(x;\omega_{{\scriptscriptstyle+}})/\sigma(X;\omega_{{\scriptscriptstyle+}})$ has a limiting distribution $\mu_{\omega_{{\scriptscriptstyle+}}}$ as detailed in Theorem $1$. Moreover, if $\lambda_{{\scriptscriptstyle1}},\ldots,\lambda_{{\scriptscriptstyle n}}$ are linearly independent over $\mathbb{Z}$, then $\mu_{\omega_{{\scriptscriptstyle+}}}$ is absolutely continuous $\textit{d}\mu_{\omega_{{\scriptscriptstyle+}}}(\alpha)=\mathcal{P}_{\Theta_{{\scriptscriptstyle+}}}(\alpha)\textit{d}\alpha$, where the density $\mathcal{P}_{\Theta_{{\scriptscriptstyle+}}}$ depends only on $\Theta_{{\scriptscriptstyle+}}$ and is otherwise independent of the $\lambda_{{\scriptscriptstyle\ell}}^{'}\text{s}$ and the exponent $A$, and is given by
\begin{equation}\label{eq:1.5}
\mathcal{P}_{\Theta_{{\scriptscriptstyle+}}}(\alpha)=\frac{\Vert\Theta_{{\scriptscriptstyle+}}\Vert_{{\scriptscriptstyle2}}}{\sqrt{2\pi}}\int\limits_{\mathbf{t}\in[0,1)^{n}}\frac{1}{\Theta_{{\scriptscriptstyle+}}(\mathbf{t})}\exp{\bigg(-\frac{1}{2}\frac{\Vert\Theta_{{\scriptscriptstyle+}}\Vert^{{\scriptscriptstyle2}}_{{\scriptscriptstyle2}}}{\Theta^{{\scriptscriptstyle2}}_{{\scriptscriptstyle+}}(\mathbf{t})}\alpha^{2}\bigg)}\textit{d}\mathbf{t}\,.
\end{equation}
\end{thm2b}
\begin{rem}
When $\lambda_{{\scriptscriptstyle1}},\ldots,\lambda_{{\scriptscriptstyle n}}$ are linearly dependent over $\mathbb{Z}$, then $\mu_{\omega_{{\scriptscriptstyle\times}}}$, $\mu_{\omega_{{\scriptscriptstyle+}}}$ do depend on the $\lambda_{{\scriptscriptstyle\ell}}^{'}\text{s}$. Moreover, one can show that both distributions are absolutely continuous and obtain generalization of \eqref{eq:1.4} and \eqref{eq:1.5} for the corresponding densities. 
%When the $\lambda_{{\scriptscriptstyle\ell}}^{'}\text{s}$ are linearly dependent over $\mathbb{Z}$ $\lambda_{{\scriptscriptstyle1}},\ldots,\lambda_{{\scriptscriptstyle n}}$. Moreover, one can show that both distributions are absolutely continuous and obtain generalization of \eqref{eq:1.4} and \eqref{eq:1.5} for corresponding densities.  
%$\lambda_{{\scriptscriptstyle1}},\ldots,\lambda_{{\scriptscriptstyle n}}$ are now assumed to be linearly dependent over $\mathbb{Z}$
%It remains an open problem to determine whether it is always the case that the limiting distribution $\mu_{{\scriptscriptstyle\omega}}$ in Theorem $1$ is absolutely continuous whenever $\omega\in\mathbf{\Omega}$. In fact, even the case of $\mu_{\omega_{{\scriptscriptstyle\times}}}$ and $\mu_{\omega_{{\scriptscriptstyle+}}}$, where $\lambda_{{\scriptscriptstyle1}},\ldots,\lambda_{{\scriptscriptstyle n}}$ are now assumed to be linearly dependent over $\mathbb{Z}$, remains open.  
\end{rem}
\begin{rem}
It remains an open problem to determine whether $\mu_{{\scriptscriptstyle\omega}}$ is absolutely continuous whenever $\omega\in\mathbf{\Omega}$.
\end{rem}
%
%Conditions \textbf{(4\textnormal{\textbf{c}})} ensures that if such a measure exists, then it is unique
\subsubsection*{\textbf{Proof outline.}} Our strategy will be to show that $\widehat{\mathcal{E}}(x;\omega)/\sigma(X;\omega)$ has finite moments of any order, and that these moments give rise to a unique probability measure which is the limiting distribution described in Theorem $1$. In $\S2.1$ we show that $\widehat{\mathcal{E}}(x;\omega)$ admits a Vorono\"{i}-type series expansion, and in $\S2.2$ we estimate the moments of the corresponding series (we do this gradually over several lemmas making use of conditions \textbf{(1)}, \textbf{(2)}, \textbf{(3\textnormal{\textbf{a}})} and \textbf{(3\textnormal{\textbf{b}})}). In $\S2.3$ we prove the finiteness of moments by using the results of $\S2.2$ together with conditions \textbf{(4\textnormal{\textbf{a}})} and \textbf{(4\textnormal{\textbf{b}})}. We begin $\S3$ by proving there exists a unique measure satisfying \eqref{eq:1.2} (we use Hamburger's Theorem together with conditions \textbf{(4\textnormal{\textbf{c}})}), and we then proceed to prove the main results.
\subsubsection*{\textbf{Notation and conventions.}} For positive quantities $X,Y$ we use Vinogradov's asymptotic notation $X\ll Y$ or the Big-$O$ notation $X=O(Y)$, to mean that $X\leq cY$ for some absolute constant $c$, and write $X\asymp Y$ if $X\ll Y\ll X$. In case the implied constant depends on some variable, we indicate this by placing a subscript. In addition, we define
\begin{equation*}
\begin{split}
&\textit{(1)}\quad r_{{\scriptscriptstyle2}}(m)=\sum_{a^{2}+b^{2}=m}1\quad;\quad\text{where the representation runs over }a,b\in\mathbb{Z}\,.\\
&\textit{(2)}\quad\mu(m)=\left\{
        \begin{array}{ll}
            1 & ;\, m=1\\\\
            (-1)^{\ell} & ;\, \text{if }m=p_{{\scriptscriptstyle1}}\cdots p_{{\scriptscriptstyle\ell}},\,\text{with } p_{{\scriptscriptstyle1}},\ldots,p_{{\scriptscriptstyle\ell}}\,\text{distinct primes}\\\\
            0 & ;\,\text{otherwise}
        \end{array}
    \right.
\end{split}
\end{equation*}
\section{Power moment estimates}
\noindent
In this section we are going to estimate the moments of $\widehat{\mathcal{E}}(x;\omega)$ (see $\S2.3$, Proposition $2$), and as a corollary we shall obtain the moments of $\widehat{\mathcal{E}}(x;\omega)/\sigma(X;\omega)$ (see $\S2.3$, Proposition $3$).
\subsection{A Vorono\"{i}-type series expansion for $\widehat{\mathcal{E}}(x;\omega)$}
\noindent
\begin{prop}
Let $\omega\in\mathbf{\Omega}$, and let $X>0$ be large. Then for $X<x<2X$ we have
\begin{equation}\label{eq:2.1}
\begin{split}
\widehat{\mathcal{E}}(x;\omega)&=\frac{2^{3/2}}{\pi}\sum_{m\,\leq\,X^{2}}\frac{r_{{\scriptscriptstyle 2}}(m)}{m}\sin{\big(\pi\sqrt{m}\omega(x)\big)}\sin{\big(\pi\sqrt{m}(2x+\omega(x))\big)}-2x^{-2}\Xi_{\psi}(x;\omega)\\
&\,\,\,\,\,+O_{\epsilon}(X^{-1+\epsilon})\,,
\end{split}
\end{equation}
for any $\epsilon>0$, where $\Xi_{\psi}(x;\omega)$ is given by
\begin{equation}\label{eq:2.2}
\Xi_{\psi}(x;\omega)=\sum_{m\,\leq\,(x+\omega(x))^{2}}r_{{\scriptscriptstyle2}}(m)\psi\Big(\big((x+\omega(x))^{4}-m^{2}\big)^{1/2}\Big)-\sum_{m\,\leq\,x^{2}}r_{{\scriptscriptstyle2}}(m)\psi\Big(\big(x^{4}-m^{2}\big)^{1/2}\Big)\,,
\end{equation}
and $\psi=t-[t]-1/2$, with $[t]=\text{max}\{m\in\mathbb{Z}: m\leq t\}$, is the sawtooth function.
\end{prop}
\noindent
For the proof we need the following result on the Vorono\"{i}-type series expansion for $\mathcal{E}_{1}(x)/x^{2}$ (see \cite{3Dgathdistribution}, $\S2$ Proposition $1$).
\begin{propp}[\cite{3Dgathdistribution}]
Let $X>0$ be large. Then on setting $\widehat{\mathcal{E}}(x)=\mathcal{E}(x)/x^{2}$, we have for $x\asymp X$
\begin{equation}\label{eq:2.3}
\begin{split}
\widehat{\mathcal{E}}(x)&=-\frac{\sqrt{2}}{\pi}\sum_{m\,\leq\,X^{2}}\frac{r_{{\scriptscriptstyle 2}}(m)}{m}\cos{\big(2\pi\sqrt{m}x\big)}-2x^{-2}\sum_{m\,\leq\,x^{2}}r_{{\scriptscriptstyle2}}(m)\psi\Big(\big(x^{4}-m^{2}\big)^{1/2}\Big)\\
&\,\,\,\,\,+O_{\epsilon}\big(X^{-1+\epsilon}\big)\,.
\end{split}
\end{equation}
\end{propp}
\noindent
We can now present the proof of Proposition $1$.
\begin{proof}(Proposition $1$). Let $\omega\in\mathbf{\Omega}$, let $X>0$ be large, and let $X<x<2X$. By the definition of $\mathcal{E}(x;\omega)$ we have
\begin{equation}\label{eq:2.4}
\begin{split}
\mathcal{E}(x;\omega)&=N(x;\omega)-\textit{vol}(\mathcal{B})\sum_{j=1}^{4}\binom{4}{j}x^{4-j}\omega^{j}(x)\\
&=N(x+\omega(x))-N(x)-\big(\textit{vol}(\mathcal{B})(x+\omega(x))^{4}-\textit{vol}(\mathcal{B})x^{4}\big)\\
&=\mathcal{E}(x+\omega(x))-\mathcal{E}(x)\,.
\end{split}
\end{equation}
Since $\widehat{\mathcal{E}}(x;\omega)=\mathcal{E}(x;\omega)/x^{2}$, it follows from \eqref{eq:2.4} that
\begin{equation}\label{eq:2.5}
\begin{split}
\widehat{\mathcal{E}}(x;\omega)&=\big(1+\omega(x)/x\big)^{2}\widehat{\mathcal{E}}(x+\omega(x))-\widehat{\mathcal{E}}(x)\\
&=\widehat{\mathcal{E}}(x+\omega(x))-\widehat{\mathcal{E}}(x)+O\big(x^{-1}|\widehat{\mathcal{E}}(x+\omega(x))|\big)\\
&=\widehat{\mathcal{E}}(x+\omega(x))-\widehat{\mathcal{E}}(x)+O_{\epsilon}\big(X^{-1+\epsilon}\big)\,,
\end{split}
\end{equation}
where the term $O_{\epsilon}\big(X^{-1+\epsilon}\big)$ in the last line above comes from estimating the RHS of \eqref{eq:2.3} trivially. Now, for $X<x<2X$ we have $X<x, x+\omega(x)<3X$, and so we may appeal to \eqref{eq:2.3} obtaining 
\begin{equation}\label{eq:2.6}
\begin{split}
\widehat{\mathcal{E}}(x+\omega(x))-\widehat{\mathcal{E}}(x)
&=-\frac{2}{\pi}\sum_{m\,\leq\,X^{2}}\frac{r_{{\scriptscriptstyle 2}}(m)}{m}\Big\{\cos{\big(2\pi\sqrt{m}(x+\omega(x))\big)}-\cos{\big(2\pi\sqrt{m}x\big)}\Big\}-2x^{-2}\Xi_{\psi}(x;\omega)\\
&+O_{\epsilon}\bigg(X^{-1+\epsilon}+x^{-3}\sum_{m\,\leq\,(x+\omega(x))^{2}}r_{{\scriptscriptstyle2}}(m)\bigg)\\
&=\frac{2^{3/2}}{\pi}\sum_{m\,\leq\,X^{2}}\frac{r_{{\scriptscriptstyle 2}}(m)}{m}\sin{\big(\pi\sqrt{m}\omega(x)\big)}\sin{\big(\pi\sqrt{m}(2x+\omega(x))\big)}-2x^{-2}\Xi_{\psi}(x;\omega)\\
&\,\,\,\,\,+O_{\epsilon}\big(X^{-1+\epsilon}\big)\,.
\end{split}
\end{equation}
Inserting \eqref{eq:2.6} into the RHS of \eqref{eq:2.5} we arrive at
\begin{equation}\label{eq:2.7}
\begin{split}
\widehat{\mathcal{E}}(x;\omega)&=\frac{2^{3/2}}{\pi}\sum_{m\,\leq\,X^{2}}\frac{r_{{\scriptscriptstyle 2}}(m)}{m}\sin{\big(\pi\sqrt{m}\omega(x)\big)}\sin{\big(\pi\sqrt{m}(2x+\omega(x))\big)}-2x^{-2}\Xi_{\psi}(x;\omega)\\
&\,\,\,\,\,+O_{\epsilon}(X^{-1+\epsilon})\,.
\end{split}
\end{equation}
This concludes the proof.
\end{proof}
\subsection{Auxiliary estimates}
In this subsection we establish several mean-value estimates that will be used in the proof of Proposition $2$ (see $\S2.3$). We begin with the following lemma showing that $\Xi_{\psi}(x;\omega)$ defined in \eqref{eq:2.2} exhibits near square-root cancellation on average.
\begin{lemma}
Let $\omega\in\mathbf{\Omega}$. Then we have
\begin{equation}\label{eq:2.8}
\frac{1}{X}\int\limits_{X}^{2X}\Xi^{2}_{\psi}(x;\omega)\textit{d}x\ll X^{2}\big(\log{X}\big)^{4}\,.
\end{equation}
\end{lemma}
\noindent
For the proof of Lemma $1$ we need the following result of Valler on series approximation for the $\psi$-function.
\begin{VL}[\cite{vaaler1985some}]
Let $H\geq1$. Then there exist trigonometrical polynomials
\begin{equation*}
\begin{split}
\psi_{{\scriptscriptstyle H}}(t)=\sum_{h\,\leq\,H}\nu(h)\sin{(2\pi ht)}\quad;\quad\psi^{\ast}_{{\scriptscriptstyle H}}(t)=\sum_{h\,\leq\,H}\nu^{\ast}(h)\cos{(2\pi ht)}\,,
\end{split}
\end{equation*}
with real coefficients satisfying $|\nu(h)|,\,|\nu^{\ast}(h)|\ll 1/h$, such that
\begin{equation*}
\big|\psi(t)-\psi_{{\scriptscriptstyle H}}(t)\big|\leq\psi^{\ast}_{{\scriptscriptstyle H}}(t)+\frac{1}{2[H]+2}\,.
\end{equation*}
\end{VL}
\begin{proof}(Lemma 1). 
Let $X>0$ be large. Let us define for $y>0$ the weighted exponential sum
\begin{equation*}
W_{h}(y)=\sum_{m\,\leq\,y}r_{{\scriptscriptstyle2}}(m)\exp{\Big(2\pi ih\big(y^{2}-m^{2}\big)^{1/2}\Big)}\,.
\end{equation*}
By Vaaler's Lemma with $H=X$, we have for $x$ in the range $X<x<2X$
\begin{equation}\label{eq:2.9}
|\Xi_{\psi}(x;\omega)|\ll\sum_{h\,\leq\,X}\frac{1}{h}\Big\{\big|W_{h}\big(x^{2}\big)\big|+\big|W_{h}\big((x+\omega(x))^{2}\big)\big|\Big\}+X\,.
\end{equation}
Applying Cauchy–Schwarz inequality we obtain
\begin{equation}\label{eq:2.10}
\Xi^{2}_{\psi}(x;\omega)\ll\big(\log{X}\big)\sum_{1\,\leq\,h\,\leq\,X}\frac{1}{h}\Big\{\big|W_{h}\big(x^{2}\big)\big|^{2}+\big|W_{h}\big((x+\omega(x))^{2}\big)\big|^{2}\Big\}+X^{2}\,.
\end{equation}
Integrating over the prescribed range, we find that
\begin{equation}\label{eq:2.11}
\frac{1}{X}\int\limits_{X}^{2X}\Xi^{2}_{\psi}(x;\omega)\textit{d}x\ll\big(\log{X}\big)\sum_{h\,\leq\,X}\frac{1}{h}\frac{1}{X}\int\limits_{X}^{2X}\Big\{\big|W_{h}\big(x^{2}\big)\big|^{2}+\big|W_{h}\big((x+\omega(x))^{2}\big)\big|^{2}\Big\}\textit{d}x+X^{2}\,.
\end{equation}
%
%\textbf{(2)}
Let us fix an integer $h\leq X$. A change of variable gives
\begin{equation}\label{eq:2.12}
\frac{1}{X}\int\limits_{X}^{2X}\big|W_{h}\big(x^{2}\big)\big|^{2}\textit{d}x\ll\frac{1}{X^{4}}\int\limits_{X^{4}}^{16X^{4}}\big|W_{h}\big(\sqrt{x}\,\big)\big|^{2}\textit{d}x\,,
\end{equation}
%
%($|\omega^{{\scriptscriptstyle(1)}}(x)|<\frac{1}{2}$ for all sufficiently large $x$)
and since $\omega$ satisfies condition \textbf{(2)} in the definition of $\mathbf{\Omega}$, we also have 
\begin{equation}\label{eq:2.13}
\begin{split}
\frac{1}{X}\int\limits_{X}^{2X}\big|W_{h}\big((x+\omega(x))^{2}\big)\big|^{2}\textit{d}x&\ll\frac{1}{X^{4}}\int\limits_{X}^{2X}\big|W_{h}\big((x+\omega(x))^{2}\big)\big|^{2}(x+\omega(x))^{3}(1+\omega^{{\scriptscriptstyle(1)}}(x))\textit{d}x\\
&\ll\frac{1}{X^{4}}\int\limits_{(X+\omega(X))^{4}}^{(2X+\omega(2X))^{4}}\big|W_{h}\big(\sqrt{x}\,\big)\big|^{2}\textit{d}x\ll\frac{1}{X^{4}}\int\limits_{X^{4}}^{81X^{4}}\big|W_{h}\big(\sqrt{x}\,\big)\big|^{2}\textit{d}x
\end{split}
\end{equation}
\label{eq:2.15}It follows from \eqref{eq:2.12} and \eqref{eq:2.13} that
\begin{equation}\label{eq:2.14}
\begin{split}
\frac{1}{X}\int\limits_{X}^{2X}\Big\{\big|W_{h}\big(x^{2}\big)\big|^{2}+\big|W_{h}\big((x+\omega(x))^{2}\big)\big|^{2}\Big\}\textit{d}x&\ll\frac{1}{X^{4}}\int\limits_{X^{4}}^{81X^{4}}\big|W_{h}\big(\sqrt{x}\,\big)\big|^{2}\textit{d}x\\
&=\sum_{m,\ell\,\leq9X^{2}}r_{{\scriptscriptstyle2}}(m)r_{{\scriptscriptstyle2}}(\ell)\mathscr{I}_{h,m,\ell}(X)\,,
\end{split}
\end{equation}
%
%\frac{1}{X^{4}}\int\limits_{\text{max}\big\{m^{2},\,n^{2},\,X^{4}\big\}}^{81X^{4}}\exp{\Big(2\pi ih\Big\{\big(x-m^{2}\big)^{1/2}-\big(x-n^{2}\big)^{1/2}\Big\}\Big)}\textit{d}x
where $\mathscr{I}_{h,m,\ell}(X)$ is given by
\begin{equation}\label{eq:2.15}
\begin{split}
\mathscr{I}_{h,m,\ell}(X)&=\frac{1}{X^{4}}\int\limits_{\text{max}\big\{m^{2},\,\ell^{2},\,X^{4}\big\}}^{81X^{4}}\exp{\Big(2\pi ih\Big\{\big(x-m^{2}\big)^{1/2}-\big(x-\ell^{2}\big)^{1/2}\Big\}\Big)}\textit{d}x\\
&\ll\left\{
        \begin{array}{ll}
            1 & ;\,m=\ell 
            \\\\
            X^{2}h^{-1}|m^{2}-\ell^{2}|^{-1} & ;\, m\neq\ell\,.
        \end{array}
    \right.
\end{split}
\end{equation}
Inserting \eqref{eq:2.15} into the RHS of \eqref{eq:2.14} we find that
\begin{equation}\label{eq:2.16}
\begin{split}
\frac{1}{X}\int\limits_{X}^{2X}\Big\{\big|W_{h}\big(x^{2}\big)\big|^{2}+\big|W_{h}\big((x+\omega(x))^{2}\big)\big|^{2}\Big\}\textit{d}x&\ll\sum_{m\,\leq\,9X^{2}}r^{2}_{{\scriptscriptstyle2}}(m)+X^{2}h^{-1}\sum_{m\neq\ell\,\leq\,9X^{2}}\frac{r_{{\scriptscriptstyle2}}(m)r_{{\scriptscriptstyle2}}(\ell)}{\big|m^{2}-\ell^{2}\big|}\\
&\ll X^{2}\log{X}+X^{2}h^{-1}\big(\log{X}\big)^{3}\,.
\end{split}
\end{equation}
Multiplying \eqref{eq:2.16} by $1/h$, summing over all $1\leq h\leq X$ and using \eqref{eq:2.11}, we obtain the desired estimate \eqref{eq:2.8}. This concludes the proof.
\end{proof}
\noindent
We now turn attention to the series expansion appearing on the RHS of \eqref{eq:2.1}, with the ultimate goal of estimating its $j$-th moment. For large $j$ however, the series contains too many terms for it to be effectively estimated, and therefore must be truncated at a suitable point. The following lemma shows that if the truncation parameter is large enough, then the tail of the series is small on average.
\begin{lemma}
Let $\omega\in\mathbf{\Omega}$, and let $Y\geq1$. Then for $X>Y^{1/2}$ we have
\begin{equation}\label{eq:2.17}
\begin{split}
&\frac{1}{X}\int\limits_{X}^{2X}\bigg|\sum_{Y<m\,\leq\,X^{2}}\frac{r_{{\scriptscriptstyle 2}}(m)}{m}\sin{\big(\pi\sqrt{m}\omega(x)\big)}\sin{\big(\pi\sqrt{m}(2x+\omega(x))\big)}\bigg|^{2}\textit{d}x\ll Y^{-1/4}(\log{X})^{2}\,.
\end{split}
\end{equation}
\end{lemma}
\begin{proof}
Expanding out the product of the $\sin{}$ functions in which
$\omega$ appears in the argument, isolating the diagonal-terms and applying integration by parts for the the off-diagonal, we find that
\begin{equation}\label{eq:2.18}
\begin{split}
&\frac{1}{X}\int\limits_{X}^{2X}\bigg|\sum_{Y<m\,\leq\,X^{2}}\frac{r_{{\scriptscriptstyle 2}}(m)}{m}\sin{\big(\pi\sqrt{m}\omega(x)\big)}\sin{\big(\pi\sqrt{m}(2x+\omega(x))\big)}\bigg|^{2}\textit{d}x\\
&=\frac{1}{2X}\int\limits_{X}^{2X}\Bigg\{\,\,\sum_{Y<\,m\,\leq\,X^{2}}\frac{r^{2}_{{\scriptscriptstyle 2}}(m)}{m^{2}}\sin^{2}{\big(\pi\sqrt{m}\omega(x)\big)}\Bigg\}\textit{d}x\\
&\,\,\,\,\,-\frac{1}{4\pi X}\sum_{e_{{\scriptscriptstyle1}},e_{{\scriptscriptstyle2}}=\pm1}e_{{\scriptscriptstyle1}}e_{{\scriptscriptstyle2}}\bigg\{\text{\Large{Q}}_{(e_{{\scriptscriptstyle1}},e_{{\scriptscriptstyle2}})\,}(Y,X)-2\pi\text{\Large{H}}_{(e_{{\scriptscriptstyle1}},e_{{\scriptscriptstyle2}})\,}(Y,X)+\text{\Large{J}}_{(e_{{\scriptscriptstyle1}},e_{{\scriptscriptstyle2}})\,}(Y,X)\bigg\}\,,
\end{split}
\end{equation}
where for $e_{{\scriptscriptstyle1}},e_{{\scriptscriptstyle2}}=\pm1$ we define
\begin{equation}\label{eq:2.19}
\begin{split}
&\text{\Large{Q}}_{(e_{{\scriptscriptstyle1}},e_{{\scriptscriptstyle2}})\,}(Y,X)=\underset{e_{{\scriptscriptstyle1}}\sqrt{m}+e_{{\scriptscriptstyle2}}\sqrt{n}\neq0}{\sum_{Y<\,m,n\,\leq\,X^{2}}}\frac{r_{{\scriptscriptstyle 2}}(m)}{m}\frac{r_{{\scriptscriptstyle 2}}(n)}{n}\frac{1}{e_{{\scriptscriptstyle1}}\sqrt{m}+e_{{\scriptscriptstyle2}}\sqrt{n}}\Big\{\text{\Large{C}}^{(\sqrt{m},\sqrt{n})\,}_{(e_{{\scriptscriptstyle1}},e_{{\scriptscriptstyle2}})\,}(2X)-\text{\Large{C}}^{(\sqrt{m},\sqrt{n})\,}_{(e_{{\scriptscriptstyle1}},e_{{\scriptscriptstyle2}})\,}(X)\Big\}\,,
\end{split}
\end{equation}
with
\begin{equation}\label{eq:2.20}
\text{\Large{C}}^{(\sqrt{m},\sqrt{n})\,}_{(e_{{\scriptscriptstyle1}},e_{{\scriptscriptstyle2}})\,}(t)=\sin{\big(\pi\sqrt{m}\omega(t)\big)}\sin{\big(\pi\sqrt{n}\omega(t)\big)}\frac{\sin{\big(\pi(e_{{\scriptscriptstyle1}}\sqrt{m}+e_{{\scriptscriptstyle2}}\sqrt{n})(2t+\omega(t))\big)}}{2+\omega^{{\scriptscriptstyle(1)}}(t)}\,,
\end{equation}
and
\begin{equation}\label{eq:2.21}
\begin{split}
\text{\Large{H}}_{(e_{{\scriptscriptstyle1}},e_{{\scriptscriptstyle2}})\,}(Y,X)=\int\limits_{X}^{2X}&\Bigg\{\underset{e_{{\scriptscriptstyle1}}\sqrt{m}+e_{{\scriptscriptstyle2}}\sqrt{n}\neq0}{\sum_{Y<\,m,n\,\leq\,X^{2}}}\frac{r_{{\scriptscriptstyle 2}}(m)}{\sqrt{m}}\frac{r_{{\scriptscriptstyle 2}}(n)}{n}\frac{1}{e_{{\scriptscriptstyle1}}\sqrt{m}+e_{{\scriptscriptstyle2}}\sqrt{n}}\sin{\big(\pi(e_{{\scriptscriptstyle1}}\sqrt{m}+e_{{\scriptscriptstyle2}}\sqrt{n})(2x+\omega(x))\big)}\\
&\,\,\times\frac{\cos{\big(\pi\sqrt{m}\omega(x)\big)}\sin{\big(\pi\sqrt{n}\omega(x)\big)}}{2+\omega^{{\scriptscriptstyle(1)}}(x)}\Bigg\}\omega^{{\scriptscriptstyle(1)}}(x)\textit{d}x\,,
\end{split}
\end{equation}
and
\begin{equation}\label{eq:2.22}
\begin{split}
\text{\Large{J}}_{(e_{{\scriptscriptstyle1}},e_{{\scriptscriptstyle2}})\,}(Y,X)=\int\limits_{X}^{2X}&\Bigg\{\underset{e_{{\scriptscriptstyle1}}\sqrt{m}+e_{{\scriptscriptstyle2}}\sqrt{n}\neq0}{\sum_{Y<\,m,n\,\leq\,X^{2}}}\frac{r_{{\scriptscriptstyle 2}}(m)}{m}\frac{r_{{\scriptscriptstyle 2}}(n)}{n}\frac{1}{e_{{\scriptscriptstyle1}}\sqrt{m}+e_{{\scriptscriptstyle2}}\sqrt{n}}\sin{\big(\pi(e_{{\scriptscriptstyle1}}\sqrt{m}+e_{{\scriptscriptstyle2}}\sqrt{n})(2x+\omega(x))\big)}\\
&\,\,\times\sin{\big(\pi\sqrt{m}\omega(x)\big)}\sin{\big(\pi\sqrt{n}\omega(x)\big)}\Bigg\}\frac{\omega^{{\scriptscriptstyle(2)}}(x)}{(2+\omega^{{\scriptscriptstyle(1)}}(x))^{2}}\textit{d}x\,,
\end{split}
\end{equation}
Let us estimate each of the summands appearing on the RHS of \eqref{eq:2.18}. For the first summand, integrating trivially, we find that
\begin{equation}\label{eq:2.23}
\begin{split}
\Bigg|\int\limits_{X}^{2X}\Bigg\{\,\,\sum_{Y<\,m\,\leq\,X^{2}}\frac{r^{2}_{{\scriptscriptstyle 2}}(m)}{m^{2}}\sin^{2}{\big(\pi\sqrt{m}\omega(x)\big)}\Bigg\}\textit{d}x\Bigg|&\ll\sum_{m\,>Y}\frac{r^{2}_{{\scriptscriptstyle 2}}(m)}{m^{2}}\\
&\ll Y^{-1}\log{2Y}\,.
\end{split}
\end{equation}
Now, fix a pair $e_{{\scriptscriptstyle1}},e_{{\scriptscriptstyle2}}=\pm1$. We estimate each of the remaining summands separately.\\
\textbf{(\textnormal{\textbf{i}})} Estimating $\text{\Large{Q}}_{(e_{{\scriptscriptstyle1}},e_{{\scriptscriptstyle2}})\,}(Y,X)$. If $e_{{\scriptscriptstyle1}}=e_{{\scriptscriptstyle2}}$, then $|e_{{\scriptscriptstyle1}}\sqrt{m}+e_{{\scriptscriptstyle2}}\sqrt{n}|=\sqrt{m}+\sqrt{n}\geq(mn)^{1/4}$, and making use of condition \textbf{(2)} we find that
\begin{equation}\label{eq:2.24}
\begin{split}
e_{{\scriptscriptstyle1}}=e_{{\scriptscriptstyle2}}:\quad\big|\text{\Large{Q}}_{(e_{{\scriptscriptstyle1}},e_{{\scriptscriptstyle2}})\,}(Y,X)\big|&\ll\bigg\{\frac{1}{|2+\omega^{{\scriptscriptstyle(1)}}(2X)|}+\frac{1}{|2+\omega^{{\scriptscriptstyle(1)}}(X)|}\bigg\}\Bigg(\sum_{m\,>Y}\frac{r_{{\scriptscriptstyle 2}}(m)}{m^{5/4}}\Bigg)^{2}\\
&\ll Y^{-1/2}\,.
\end{split}
\end{equation}
If $e_{{\scriptscriptstyle1}}\neq e_{{\scriptscriptstyle2}}$, then defining $\lambda_{{\scriptscriptstyle m}}$ for a positive integer $m$ to be $\lambda_{{\scriptscriptstyle m}}=\min_{m\neq n\in\mathbb{N}}|\sqrt{m}-\sqrt{n}|$, by Hilbert's inequality \cite{montgomery1974hilbert} combined with condition \textbf{(2)} we find that
\begin{equation}\label{eq:2.25}
\begin{split}
e_{{\scriptscriptstyle1}}\neq e_{{\scriptscriptstyle2}}:\quad\big|\text{\Large{Q}}_{(e_{{\scriptscriptstyle1}},e_{{\scriptscriptstyle2}})\,}(Y,X)\big|&\ll\bigg\{\frac{1}{|2+\omega^{{\scriptscriptstyle(1)}}(2X)|}+\frac{1}{|2+\omega^{{\scriptscriptstyle(1)}}(X)|}\bigg\}\sum_{m\,>Y}\frac{r^{2}_{{\scriptscriptstyle 2}}(m)}{m^{2}}\lambda^{-1}_{{\scriptscriptstyle m}}\\
&\ll\sum_{m\,>Y}r^{2}_{{\scriptscriptstyle2}}(m)m^{-3/2}\\
&\ll Y^{-1/2}\log{2Y}\,.
\end{split}
\end{equation}
\textbf{(\textnormal{\textbf{ii}})} Estimating $\text{\Large{H}}_{(e_{{\scriptscriptstyle1}},e_{{\scriptscriptstyle2}})\,}(Y,X)$. According to the assumptions on the gap width function $\omega$, we may find a finite collection of disjoint intervals $\{I_{\alpha}\}_{\alpha\in\mathcal{A}}$ with $|\mathcal{A}|=1+|\mathcal{U}^{\omega}_{{\scriptscriptstyle X}}|$ satisfying $\biguplus_{\alpha\in\mathcal{A}}I_{\alpha}=[X,2X]$, and such that $\omega^{{\scriptscriptstyle(1)}}$ does not change sign on each such interval. Combined with the first mean value theorem for integrals, we have the following decomposition
\begin{equation}\label{eq:2.26}
\text{\Large{H}}_{(e_{{\scriptscriptstyle1}},e_{{\scriptscriptstyle2}})\,}(Y,X)=\sum_{\alpha\in{\mathcal{A}}}\text{\Large{S}}_{(e_{{\scriptscriptstyle1}},e_{{\scriptscriptstyle2}},\alpha)\,}(Y,X)\int\limits_{I_{\alpha}}\omega^{{\scriptscriptstyle(1)}}(x)\textit{d}x\,,
\end{equation}
where $\text{\Large{S}}_{(e_{{\scriptscriptstyle1}},e_{{\scriptscriptstyle2}},\alpha)\,}(Y,X)$ is given by
\begin{equation}\label{eq:2.27}
\begin{split}
\text{\Large{S}}_{(e_{{\scriptscriptstyle1}},e_{{\scriptscriptstyle2}},\alpha)\,}(Y,X)=\underset{e_{{\scriptscriptstyle1}}\sqrt{m}+e_{{\scriptscriptstyle2}}\sqrt{n}\neq0}{\sum_{Y<\,m,n\,\leq\,X^{2}}}&\frac{r_{{\scriptscriptstyle 2}}(m)}{\sqrt{m}}\frac{r_{{\scriptscriptstyle 2}}(n)}{n}\frac{1}{e_{{\scriptscriptstyle1}}\sqrt{m}+e_{{\scriptscriptstyle2}}\sqrt{n}}\sin{\big(\pi(e_{{\scriptscriptstyle1}}\sqrt{m}+e_{{\scriptscriptstyle2}}\sqrt{n})(2x_{{\scriptscriptstyle\alpha}}+\omega(x_{{\scriptscriptstyle\alpha}}))\big)}\\
&\,\,\times\frac{\cos{\big(\pi\sqrt{m}\omega(x_{{\scriptscriptstyle\alpha}})\big)}\sin{\big(\pi\sqrt{n}\omega(x_{{\scriptscriptstyle\alpha}})\big)}}{2+\omega^{{\scriptscriptstyle(1)}}(x_{{\scriptscriptstyle\alpha}})}\quad;\quad x_{{\scriptscriptstyle\alpha}}\in I_{\alpha}\,.
\end{split}
\end{equation}
As before, we treat the cases $e_{{\scriptscriptstyle1}}=e_{{\scriptscriptstyle2}}$ and $e_{{\scriptscriptstyle1}}\neq e_{{\scriptscriptstyle2}}$ separately. Making use of condition \textbf{(2)} We have
\begin{equation}\label{eq:2.28}
\begin{split}
e_{{\scriptscriptstyle1}}= e_{{\scriptscriptstyle2}}:\quad\big|\text{\Large{S}}_{(e_{{\scriptscriptstyle1}},e_{{\scriptscriptstyle2}},\alpha)\,}(Y,X)\big|&\ll\frac{1}{|2+\omega^{{\scriptscriptstyle(1)}}(x_{{\scriptscriptstyle\alpha}})|}\sum_{Y<\,m,n\,\leq\,X^{2}}\frac{r_{{\scriptscriptstyle 2}}(m)}{\sqrt{m}}\frac{r_{{\scriptscriptstyle 2}}(n)}{n}\frac{1}{\sqrt{m}+\sqrt{n}}\Big\{\mathds{1}_{n\,\leq\,m}+\mathds{1}_{n\,>\,m}\Big\}\\
&\ll(\log{X})^{2}\,.
\end{split}
\end{equation}
If $e_{{\scriptscriptstyle1}}\neq e_{{\scriptscriptstyle2}}$, then defining $\lambda_{{\scriptscriptstyle m}}$ as before, by Hilbert's and condition \textbf{(2)} we have
\begin{equation}\label{eq:2.29}
\begin{split}
e_{{\scriptscriptstyle1}}\neq e_{{\scriptscriptstyle2}}:\quad\big|\text{\Large{S}}_{(e_{{\scriptscriptstyle1}},e_{{\scriptscriptstyle2}},\alpha)\,}(Y,X)\big|&\ll\frac{1}{|2+\omega^{{\scriptscriptstyle(1)}}(x_{{\scriptscriptstyle\alpha}})|}\Bigg(\,\,\sum_{Y<\,m\,\leq\,X^{2}}\frac{r^{2}_{{\scriptscriptstyle 2}}(m)}{m}\lambda^{-1}_{{\scriptscriptstyle m}}\Bigg)^{1/2}\Bigg(\,\,\sum_{Y<\,m\,\leq\,X^{2}}\frac{r^{2}_{{\scriptscriptstyle 2}}(m)}{m^{2}}\lambda^{-1}_{{\scriptscriptstyle m}}\Bigg)^{1/2}\\
&\ll\Bigg(\,\,\sum_{Y<\,m\,\leq\,X^{2}}\frac{r^{2}_{{\scriptscriptstyle 2}}(m)}{\sqrt{m}}\Bigg)^{1/2}\Bigg(\,\,\sum_{Y<\,m\,\leq\,X^{2}}\frac{r^{2}_{{\scriptscriptstyle 2}}(m)}{m^{3/2}}\Bigg)^{1/2}\\
&\ll X^{1/2}Y^{-1/4}\log{X}\,.
\end{split}
\end{equation}
Inserting the estimates \eqref{eq:2.28} and \eqref{eq:2.29} into the RHS of \eqref{eq:2.26}, and making use of condition \textbf{(3\textnormal{\textbf{a}})}, we find that
\begin{equation}\label{eq:2.30}
\begin{split}
e_{{\scriptscriptstyle1}}=e_{{\scriptscriptstyle2}}:\quad\big|\text{\Large{H}}_{(e_{{\scriptscriptstyle1}},e_{{\scriptscriptstyle2}})\,}(Y,X)\big|&\ll\sum_{\alpha\in{\mathcal{A}}}\big|\text{\Large{S}}_{(e_{{\scriptscriptstyle1}},e_{{\scriptscriptstyle2}},\alpha)\,}(Y,X)\big|\bigg|\int\limits_{I_{\alpha}}\omega^{{\scriptscriptstyle(1)}}(x)\textit{d}x\bigg|\\
&\ll\Big(|\mathcal{A}|\max\limits_{X\leq x\leq2X}\omega(x)\Big)(\log{X})^{2}\\
&=\Big(\big(1+|\mathcal{U}^{\omega}_{{\scriptscriptstyle X}}|\big)\max\limits_{X\leq x\leq2X}\omega(x)\Big)(\log{X})^{2}\\
&\ll X^{1/2}(\log{X})^{2}\,.
\end{split}
\end{equation}
and
\begin{equation}\label{eq:2.31}
\begin{split}
e_{{\scriptscriptstyle1}}\neq e_{{\scriptscriptstyle2}}:\quad\big|\text{\Large{H}}_{(e_{{\scriptscriptstyle1}},e_{{\scriptscriptstyle2}})\,}(Y,X)\big|&\ll\sum_{\alpha\in{\mathcal{A}}}\big|\text{\Large{S}}_{(e_{{\scriptscriptstyle1}},e_{{\scriptscriptstyle2}},\alpha)\,}(Y,X)\big|\bigg|\int\limits_{I_{\alpha}}\omega^{{\scriptscriptstyle(1)}}(x)\textit{d}x\bigg|\\
&\ll\Big(|\mathcal{A}|\max\limits_{X\leq x\leq2X}\omega(x)\Big)X^{1/2}Y^{-1/4}\log{X}\\
&=\Big(\big(1+|\mathcal{U}^{\omega}_{{\scriptscriptstyle X}}|\big)\max\limits_{X\leq x\leq2X}\omega(x)\Big)X^{1/2}Y^{-1/4}\log{X}\\
&\ll XY^{-1/4}\log{X}\,.
\end{split}
\end{equation}
\textbf{(\textnormal{\textbf{iii}})} Estimating $\text{\Large{J}}_{(e_{{\scriptscriptstyle1}},e_{{\scriptscriptstyle2}})\,}(Y,X)$. According to the assumptions on the gap width function $\omega$, we may find a finite collection of disjoint intervals $\{\Upsilon_{\nu}\}_{\nu\in\mathcal{D}}$ with $|\mathcal{D}|=1+|\mathcal{V}^{\omega}_{{\scriptscriptstyle X}}|$ satisfying $\biguplus_{\nu\in\mathcal{D}}\Upsilon_{\nu}=[X,2X]$, and such that $\omega^{{\scriptscriptstyle(2)}}$ does not change sign on each such interval. Combined with the first mean value theorem for integrals, we have the following decomposition
\begin{equation}\label{eq:2.32}
\text{\Large{J}}_{(e_{{\scriptscriptstyle1}},e_{{\scriptscriptstyle2}})\,}(Y,X)=\sum_{\nu\in{\mathcal{D}}}\text{\Large{K}}_{(e_{{\scriptscriptstyle1}},e_{{\scriptscriptstyle2}},\nu)\,}(Y,X)\int\limits_{\Upsilon_{\nu}}\frac{\omega^{{\scriptscriptstyle(2)}}(x)}{(2+\omega^{{\scriptscriptstyle(1)}}(x))^{2}}\textit{d}x\,,
\end{equation}
where
\begin{equation}\label{eq:2.33}
\begin{split}
\text{\Large{K}}_{(e_{{\scriptscriptstyle1}},e_{{\scriptscriptstyle2}},v)\,}(Y,X)=\underset{e_{{\scriptscriptstyle1}}\sqrt{m}+e_{{\scriptscriptstyle2}}\sqrt{n}\neq0}{\sum_{Y<\,m,n\,\leq\,X^{2}}}&\frac{r_{{\scriptscriptstyle 2}}(m)}{m}\frac{r_{{\scriptscriptstyle 2}}(n)}{n}\frac{1}{e_{{\scriptscriptstyle1}}\sqrt{m}+e_{{\scriptscriptstyle2}}\sqrt{n}}\sin{\big(\pi(e_{{\scriptscriptstyle1}}\sqrt{m}+e_{{\scriptscriptstyle2}}\sqrt{n})(2x_{{\scriptscriptstyle\nu}}+\omega(x_{{\scriptscriptstyle\nu}}))\big)}\\
&\,\,\times\sin{\big(\pi\sqrt{m}\omega(x_{{\scriptscriptstyle\nu}})\big)}\sin{\big(\pi\sqrt{n}\omega(x_{{\scriptscriptstyle\nu}})\big)}\quad;\quad x_{{\scriptscriptstyle\nu}}\in\Upsilon_{\nu}\,.
%\textit{d}\bigg(\frac{\sin{\big(\pi\sqrt{m}\omega(x)\big)}\sin{\big(\pi\sqrt{n}\omega(x)\big)}}{2+\omega^{{\scriptscriptstyle(1)}}(x)}\bigg)\,.
\end{split}
\end{equation}
Repeating the exact same arguments as in the previous cases, we have
\begin{equation}\label{eq:2.34}
e_{{\scriptscriptstyle1}}=e_{{\scriptscriptstyle2}}:\quad\big|\text{\Large{K}}_{(e_{{\scriptscriptstyle1}},e_{{\scriptscriptstyle2}},\nu)\,}(Y,X)\big|\ll Y^{-1/2}\,,
\end{equation}
and
\begin{equation}\label{eq:2.35}
e_{{\scriptscriptstyle1}}\neq e_{{\scriptscriptstyle2}}:\quad\big|\text{\Large{K}}_{(e_{{\scriptscriptstyle1}},e_{{\scriptscriptstyle2}},\nu)\,}(Y,X)\big|\ll Y^{-1/2}\log{2Y}\,.
\end{equation}
Inserting the estimates \eqref{eq:2.34} and \eqref{eq:2.35} into the RHS of \eqref{eq:2.32}, and making use of condition \textbf{(2)} and condition \textbf{(3\textnormal{\textbf{b}})}, we find that
\begin{equation}\label{eq:3.20}
\begin{split}
\max_{e_{{\scriptscriptstyle1}},e_{{\scriptscriptstyle2}}=\pm1}\big|\text{\Large{J}}_{(e_{{\scriptscriptstyle1}},e_{{\scriptscriptstyle2}})\,}(Y,X)\big|&\ll\sum_{\nu\in{\mathcal{D}}}\max_{e_{{\scriptscriptstyle1}},e_{{\scriptscriptstyle2}}=\pm1}\big|\text{\Large{K}}_{(e_{{\scriptscriptstyle1}},e_{{\scriptscriptstyle2}},\nu)\,}(Y,X)\big|\bigg|\int\limits_{\Upsilon_{v}}\frac{\omega^{{\scriptscriptstyle(2)}}(x)}{(2+\omega^{{\scriptscriptstyle(1)}}(x))^{2}}\textit{d}x\bigg|\\
&\ll|\mathcal{D}|Y^{-1/2}\log{2Y}\\
&=\big(1+|\mathcal{V}^{\omega}_{{\scriptscriptstyle X}}|\big)Y^{-1/2}\log{2Y}\\
&\ll X^{1-\xi_{{\scriptscriptstyle\omega}}}Y^{-1/2}\log{2Y}\,.
\end{split}
\end{equation}
Inspecting the above estimates, and recalling that $Y<X^{2}$, we find that \eqref{eq:2.30} and \eqref{eq:2.31} dominate. Inserting these estimate into the RHS of \eqref{eq:2.18} we obtain
\begin{equation}\label{eq:2.37}
\begin{split}
\frac{1}{X}\int\limits_{X}^{2X}\bigg|&\sum_{Y<m\,\leq\,X^{2}}\frac{r_{{\scriptscriptstyle 2}}(m)}{m}\sin{\big(\pi\sqrt{m}\omega(x)\big)}\sin{\big(\pi\sqrt{m}(2x+\omega(x))\big)}\bigg|^{2}\textit{d}x\\
&\ll\frac{1}{X}\Bigg|\int\limits_{X}^{2X}\Bigg\{\,\,\sum_{Y<\,m\,\leq\,X^{2}}\frac{r^{2}_{{\scriptscriptstyle 2}}(m)}{m^{2}}\sin^{2}{\big(\pi\sqrt{m}\omega(x)\big)}\Bigg\}\textit{d}x\Bigg|\\
&\,\,\,\,\,\,+\frac{1}{X}\sum_{e_{{\scriptscriptstyle1}},e_{{\scriptscriptstyle2}}=\pm1}\bigg\{\Big|\text{\Large{Q}}_{(e_{{\scriptscriptstyle1}},e_{{\scriptscriptstyle2}})\,}(Y,X)\Big|+\Big|\text{\Large{H}}_{(e_{{\scriptscriptstyle1}},e_{{\scriptscriptstyle2}})\,}(Y,X)\Big|+\Big|\text{\Large{J}}_{(e_{{\scriptscriptstyle1}},e_{{\scriptscriptstyle2}})\,}(Y,X)\Big|\bigg\}\\
&\ll\frac{1}{X}\Big\{X^{1/2}(\log{X})^{2}+XY^{-1/4}\log{X}\Big\}\\
&\ll Y^{-1/4}(\log{X})^{2}\,.
\end{split}
\end{equation}
This concludes the proof.
\end{proof}
\noindent
We will need the following lemma when we come to estimate the $j$-th moment of the truncated part of the series expansion of $\widehat{\mathcal{E}}(x;\omega)$. The lemma essentially says that if the truncation parameter is not too large, then the non-diagonal terms make a negligible contribution. 
\begin{lemma}
Let $\omega\in\mathbf{\Omega}$, and let $Y\geq1$. Let $\mathfrak{f}(t)=\cos{t},\,\sin{t}$, and let $e_{{\scriptscriptstyle1}},\ldots,e_{{\scriptscriptstyle j}}=\pm1$. Then we have
\begin{equation}\label{eq:2.38}
\begin{split}
\Bigg|\frac{1}{X}\int\limits_{X}^{2X}&\underset{\sum_{i=1}^{j}e_{{\scriptscriptstyle i}}\sqrt{m_{{\scriptscriptstyle i}}}\neq0}{\sum_{m_{{\scriptscriptstyle1}},\ldots,m_{{\scriptscriptstyle j}}\,\leq\,Y}}\bigg\{\prod_{i=1}^{j}\frac{r_{{\scriptscriptstyle 2}}(m_{{\scriptscriptstyle i}})}{m_{{\scriptscriptstyle i}}}\sin{\big(\pi\sqrt{m_{{\scriptscriptstyle i}}}\omega(x)\big)}\bigg\}\mathfrak{f}{\bigg(\pi\Big(\sum_{i=1}^{j}e_{{\scriptscriptstyle i}}\sqrt{m_{{\scriptscriptstyle i}}}\Big)(2x+\omega(x))\bigg)}\textit{d}x\Bigg|\\
&\ll_{j}\left\{
        \begin{array}{ll}
            X^{-1/2}\log{2Y}+X^{-\xi_{{\scriptscriptstyle\omega}}}& ;\,j=1
            \\\\
            \big(X^{-1/2}Y^{1/2}+X^{-\xi_{{\scriptscriptstyle\omega}}}\big)Y^{2^{j-2}-3/2}(\log{2Y})^{j-1}& ;\,j\geq2\,.
        \end{array}
    \right.
\end{split}
\end{equation}
\end{lemma}
\begin{proof}
Let $\mathfrak{g}(t)=\cos{t},\,-\sin{t}$ be the primitive function of $\mathfrak{f}$. We decompose the integral in question using the same procedure as in the proof of Lemma $2$. Let $\{I_{\alpha}\}_{\alpha\in\mathcal{A}}$ and $\{\Upsilon_{\nu}\}_{\nu\in\mathcal{D}}$ be collection of intervals with $|\mathcal{A}|=1+|\mathcal{U}^{\omega}_{{\scriptscriptstyle X}}|$ and $|\mathcal{D}|=1+|\mathcal{V}^{\omega}_{{\scriptscriptstyle X}}|$, satisfying $\biguplus_{\alpha\in\mathcal{A}}I_{\alpha}=[X,2X]$ and $\biguplus_{\nu\in\mathcal{D}}\Upsilon_{\nu}=[X,2X]$, such that $\omega^{{\scriptscriptstyle(1)}}$ and $\omega^{{\scriptscriptstyle(2)}}$ do not change sign on $I_{\alpha}$ and $\Upsilon_{\nu}$ respectively. For integers $m_{{\scriptscriptstyle1}},\ldots,m_{{\scriptscriptstyle j}}\geq1$, let
\begin{equation*}
\Psi^{(\sqrt{m_{{\scriptscriptstyle1}}},\ldots\sqrt{m_{{\scriptscriptstyle j}}})\,}_{(e_{{\scriptscriptstyle1}},\ldots e_{{\scriptscriptstyle j}})\,}(t)=\mathfrak{g}\bigg(\pi\Big(\sum_{i=1}^{j}e_{{\scriptscriptstyle i}}\sqrt{m_{{\scriptscriptstyle i}}}\Big)(2t+\omega(t))\bigg)\frac{\prod_{i=1}^{j}\sin{\big(\pi\sqrt{m_{{\scriptscriptstyle i}}}\omega(t)\big)}}{2+\omega^{{\scriptscriptstyle(1)}}(t)}\,,
\end{equation*}
\begin{equation*}
\digamma^{(\sqrt{m_{{\scriptscriptstyle1}}},\ldots\sqrt{m_{{\scriptscriptstyle j}}})\,}_{(e_{{\scriptscriptstyle1}},\ldots e_{{\scriptscriptstyle j}})\,}(t)=\mathfrak{g}\bigg(\pi\Big(\sum_{i=1}^{j}e_{{\scriptscriptstyle i}}\sqrt{m_{{\scriptscriptstyle i}}}\Big)(2t+\omega(t))\bigg)\sum_{k=1}^{j}\frac{\sqrt{m_{{\scriptscriptstyle k}}}\cos{\big(\pi\sqrt{m_{{\scriptscriptstyle k}}}\omega(t)\big)}}{2+\omega^{{\scriptscriptstyle(1)}}(t)}\underset{i\neq k}{\prod_{i=1}^{j}}\sin{\big(\pi\sqrt{m_{{\scriptscriptstyle i}}}\omega(t)\big)}\,,
\end{equation*}
\begin{equation*}
\Lambda^{(\sqrt{m_{{\scriptscriptstyle1}}},\ldots\sqrt{m_{{\scriptscriptstyle j}}})\,}_{(e_{{\scriptscriptstyle1}},\ldots e_{{\scriptscriptstyle j}})\,}(t)=\mathfrak{g}\bigg(\pi\Big(\sum_{i=1}^{j}e_{{\scriptscriptstyle i}}\sqrt{m_{{\scriptscriptstyle i}}}\Big)(2t+\omega(t))\bigg)\prod_{i=1}^{j}\sin{\big(\pi\sqrt{m_{{\scriptscriptstyle i}}}\omega(t)\big)}\,.
\end{equation*}
Integrating by parts once, combined with the first mean value theorem for integrals, we have the following decomposition
\begin{equation}\label{eq:2.39}
\begin{split}
&\frac{1}{X}\int\limits_{X}^{2X}\underset{\sum_{i=1}^{j}e_{{\scriptscriptstyle i}}\sqrt{m_{{\scriptscriptstyle i}}}\neq0}{\sum_{m_{{\scriptscriptstyle1}},\ldots,m_{{\scriptscriptstyle j}}\,\leq\,Y}}\bigg\{\prod_{i=1}^{j}\frac{r_{{\scriptscriptstyle 2}}(m_{{\scriptscriptstyle i}})}{m_{{\scriptscriptstyle i}}}\sin{\big(\pi\sqrt{m_{{\scriptscriptstyle i}}}\omega(x)\big)}\bigg\}\mathfrak{f}{\bigg(\pi\Big(\sum_{i=1}^{j}e_{{\scriptscriptstyle i}}\sqrt{m_{{\scriptscriptstyle i}}}\Big)(2x+\omega(x))\bigg)}\textit{d}x\\
&=\frac{1}{\pi X}\underset{\sum_{i=1}^{j}e_{{\scriptscriptstyle i}}\sqrt{m_{{\scriptscriptstyle i}}}\neq0}{\sum_{m_{{\scriptscriptstyle1}},\ldots,m_{{\scriptscriptstyle j}}\,\leq\,Y}}\Bigg(\frac{1}{\sum_{i=1}^{j}e_{{\scriptscriptstyle i}}\sqrt{m_{{\scriptscriptstyle i}}}}\prod_{i=1}^{j}\frac{r_{{\scriptscriptstyle 2}}(m_{{\scriptscriptstyle i}})}{m_{{\scriptscriptstyle i}}}\Bigg)\Bigg\{\Psi^{(\sqrt{m_{{\scriptscriptstyle1}}},\ldots\sqrt{m_{{\scriptscriptstyle j}}})\,}_{(e_{{\scriptscriptstyle1}},\ldots e_{{\scriptscriptstyle j}})\,}(2X)-\Psi^{(\sqrt{m_{{\scriptscriptstyle1}}},\ldots\sqrt{m_{{\scriptscriptstyle j}}})\,}_{(e_{{\scriptscriptstyle1}},\ldots e_{{\scriptscriptstyle j}})\,}(X)\\
&\,\,\,\,\,-\pi\sum_{\alpha\in{\mathcal{A}}}\digamma^{(\sqrt{m_{{\scriptscriptstyle1}}},\ldots\sqrt{m_{{\scriptscriptstyle j}}})\,}_{(e_{{\scriptscriptstyle1}},\ldots e_{{\scriptscriptstyle j}})\,}(y_{{\scriptscriptstyle\alpha}})\int\limits_{I_{\alpha}}\omega^{{\scriptscriptstyle(1)}}(x)\textit{d}x+\sum_{\nu\in{\mathcal{D}}}\Lambda^{(\sqrt{m_{{\scriptscriptstyle1}}},\ldots\sqrt{m_{{\scriptscriptstyle j}}})\,}_{(e_{{\scriptscriptstyle1}},\ldots e_{{\scriptscriptstyle j}})\,}(y_{{\scriptscriptstyle\nu}})\int\limits_{\Upsilon_{\nu}}\frac{\omega^{{\scriptscriptstyle(2)}}(x)}{(2+\omega^{{\scriptscriptstyle(1)}}(x))^{2}}\textit{d}x\Bigg\}\,,
\end{split}
\end{equation}
where $y_{{\scriptscriptstyle\alpha}}\in I_{\alpha}$ and $y_{{\scriptscriptstyle\nu}}\in\Upsilon_{\nu}$. We have according to condition \textbf{(2)}
\begin{equation}\label{eq:2.40}
\big|\Psi^{(\sqrt{m_{{\scriptscriptstyle1}}},\ldots\sqrt{m_{{\scriptscriptstyle j}}})\,}_{(e_{{\scriptscriptstyle1}},\ldots e_{{\scriptscriptstyle j}})\,}(2X)-\Psi^{(\sqrt{m_{{\scriptscriptstyle1}}},\ldots\sqrt{m_{{\scriptscriptstyle j}}})\,}_{(e_{{\scriptscriptstyle1}},\ldots e_{{\scriptscriptstyle j}})\,}(X)\big|\ll1\,.
\end{equation}
According to condition \textbf{(2)} and condition \textbf{(3\textnormal{\textbf{a}})} we have
\begin{equation}\label{eq:2.41}
\begin{split}
\bigg|\sum_{\alpha\in{\mathcal{A}}}\digamma^{(\sqrt{m_{{\scriptscriptstyle1}}},\ldots\sqrt{m_{{\scriptscriptstyle j}}})\,}_{(e_{{\scriptscriptstyle1}},\ldots e_{{\scriptscriptstyle j}})\,}(y_{{\scriptscriptstyle\alpha}})\int\limits_{I_{\alpha}}\omega^{{\scriptscriptstyle(1)}}(x)\textit{d}x\bigg|&\ll_{j}\Big(|\mathcal{A}|\max\limits_{X\leq x\leq2X}\omega(x)\Big)\underset{1\leq\,k\,\leq j}{\max}\sqrt{m_{{\scriptscriptstyle k}}}\\
&=\Big(\big(1+|\mathcal{U}^{\omega}_{{\scriptscriptstyle X}}|\big)\max\limits_{X\leq x\leq2X}\omega(x)\Big)\underset{1\leq\,k\,\leq j}{\max}\sqrt{m_{{\scriptscriptstyle k}}}\\
&\ll X^{1/2}\underset{1\leq\,k\,\leq j}{\max}\sqrt{m_{{\scriptscriptstyle k}}}\,,
\end{split}
\end{equation}
and according to condition \textbf{(2)}  and condition \textbf{(3\textnormal{\textbf{b}})} we have
\begin{equation}\label{eq:2.42}
\begin{split}
\bigg|\sum_{\nu\in{\mathcal{D}}}\Lambda^{(\sqrt{m_{{\scriptscriptstyle1}}},\ldots\sqrt{m_{{\scriptscriptstyle j}}})\,}_{(e_{{\scriptscriptstyle1}},\ldots e_{{\scriptscriptstyle j}})\,}(y_{{\scriptscriptstyle\nu}})\int\limits_{\Upsilon_{\nu}}\frac{\omega^{{\scriptscriptstyle(2)}}(x)}{(2+\omega^{{\scriptscriptstyle(1)}}(x))^{2}}\textit{d}x\bigg|&\ll|\mathcal{D}|\\
&=1+|\mathcal{V}^{\omega}_{{\scriptscriptstyle X}}|\\
&\ll X^{1-\xi_{{\scriptscriptstyle\omega}}}\,.
\end{split}
\end{equation}
%
%\textbf{(\textnormal{\textbf{i}})}
We now proceed to prove \eqref{eq:2.38}, and we do so according the whether $j=1,2$ or $j\geq3$.\\
\textbf{\textnormal{\textbf{Case (I)}}}: $j=1$. In this case the restriction $\pm\sqrt{m}\neq0$ is redundant, and it follows from \eqref{eq:2.39} that
\begin{equation}\label{eq:2.43}
\begin{split}
\Bigg|\frac{1}{X}\int\limits_{X}^{2X}\sum_{m\,\leq\,Y}&\bigg\{\frac{r_{{\scriptscriptstyle 2}}(m)}{m}\sin{\big(\pi\sqrt{m}\omega(x)\big)}\bigg\}\mathfrak{f}{\bigg(\pi\Big(\pm\sqrt{m}\Big)(2x+\omega(x))\bigg)}\textit{d}x\Bigg|\\
&\ll\sum_{m\,\leq Y}\frac{r_{{\scriptscriptstyle 2}}(m)}{m^{3/2}}\bigg\{X^{-1/2}m^{1/2}+X^{-\xi_{{\scriptscriptstyle\omega}}}\bigg\}\\
&\ll X^{-1/2}\log{2Y}+X^{-\xi_{{\scriptscriptstyle\omega}}}\,.
\end{split}
\end{equation}
\textbf{\textnormal{\textbf{Case (II)}}}: $j=2$. We have $e_{{\scriptscriptstyle1}}\sqrt{m_{{\scriptscriptstyle1}}}+e_{{\scriptscriptstyle2}}\sqrt{m_{{\scriptscriptstyle2}}}\neq0\Rightarrow\frac{1}{|e_{{\scriptscriptstyle1}}\sqrt{m_{{\scriptscriptstyle1}}}+e_{{\scriptscriptstyle2}}\sqrt{m_{{\scriptscriptstyle2}}}|}\ll\max_{i=1,2}\sqrt{m_{{\scriptscriptstyle i}}}$, and it follows from \eqref{eq:2.39} that
\begin{equation}\label{eq:2.44}
\begin{split}
\Bigg|\frac{1}{X}\int\limits_{X}^{2X}\underset{e_{{\scriptscriptstyle1}}\sqrt{m_{{\scriptscriptstyle1}}}+e_{{\scriptscriptstyle2}}\sqrt{m_{{\scriptscriptstyle2}}}\neq0}{\sum_{m_{{\scriptscriptstyle1}},m_{{\scriptscriptstyle2}}\,\leq\,Y}}&\bigg\{\prod_{i=1}^{2}\frac{r_{{\scriptscriptstyle 2}}(m_{{\scriptscriptstyle i}})}{m_{{\scriptscriptstyle i}}}\sin{\big(\pi\sqrt{m_{{\scriptscriptstyle i}}}\omega(x)\big)}\bigg\}\mathfrak{f}{\bigg(\pi\Big(e_{{\scriptscriptstyle1}}\sqrt{m_{{\scriptscriptstyle1}}}+e_{{\scriptscriptstyle2}}\sqrt{m_{{\scriptscriptstyle2}}}\Big)(2x+\omega(x))\bigg)}\textit{d}x\Bigg|\\
&\ll\sum_{m_{{\scriptscriptstyle1}},m_{{\scriptscriptstyle2}}\,\leq\,Y}\Big(\underset{i=1,2}{\max}\sqrt{m_{{\scriptscriptstyle i}}}\Big)\Bigg(\prod_{i=1}^{2}\frac{r_{{\scriptscriptstyle 2}}(m_{{\scriptscriptstyle i}})}{m_{{\scriptscriptstyle i}}}\bigg)\bigg\{X^{-1/2}\underset{i=1,2}{\max}\sqrt{m_{{\scriptscriptstyle i}}}+
X^{-\xi_{{\scriptscriptstyle\omega}}}\bigg\}\\
&\ll X^{-1/2}Y\log{2Y}+X^{-\xi_{{\scriptscriptstyle\omega}}}Y^{1/2}\log{2Y}\,.
\end{split}
\end{equation}
\textbf{\textnormal{\textbf{Case (III)}}}: $j\geq3$. To handle this case we need the following result (see \cite{HughesRudnick}, Lemma $7$): For positive integers $m_{1},\ldots,m_{j}$ positive integers and $\textit{e}_{1},\ldots,\textit{e}_{j}=\pm1$ it holds
\begin{equation*}
\sum_{i=1}^{j}\textit{e}_{i}\sqrt{m_{i}}\neq0\Rightarrow\bigg|\sum_{i=1}^{j}\textit{e}_{i}\sqrt{m_{i}}\bigg|\gg_{j}\underset{1\leq\,i\,\leq j}{\max}m_{{\scriptscriptstyle i}}^{1/2-2^{j-2}}\,.
\end{equation*}
It then follows from \eqref{eq:2.39} that
\begin{equation}\label{eq:2.45}
\begin{split}
\Bigg|\frac{1}{X}\int\limits_{X}^{2X}&\underset{\sum_{i=1}^{j}e_{{\scriptscriptstyle i}}\sqrt{m_{{\scriptscriptstyle i}}}\neq0}{\sum_{m_{{\scriptscriptstyle1}},\ldots,m_{{\scriptscriptstyle j}}\,\leq\,Y}}\bigg\{\prod_{i=1}^{j}\frac{r_{{\scriptscriptstyle 2}}(m_{{\scriptscriptstyle i}})}{m_{{\scriptscriptstyle i}}}\sin{\big(\pi\sqrt{m_{{\scriptscriptstyle i}}}\omega(x)\big)}\bigg\}\mathfrak{f}{\bigg(\pi\Big(\sum_{i=1}^{j}e_{{\scriptscriptstyle i}}\sqrt{m_{{\scriptscriptstyle i}}}\Big)(2x+\omega(x))\bigg)}\textit{d}x\Bigg|\\
&\ll_{j}\sum_{m_{{\scriptscriptstyle1}},\ldots,m_{{\scriptscriptstyle j}}\,\leq\,Y}\Big(\underset{1\leq\,i\,\leq j}{\max}m_{{\scriptscriptstyle i}}^{2^{j-2}-1/2}\Big)\Bigg(\prod_{i=1}^{j}\frac{r_{{\scriptscriptstyle 2}}(m_{{\scriptscriptstyle i}})}{m_{{\scriptscriptstyle i}}}\bigg)\bigg\{X^{-1/2}\underset{1\leq\,i\,\leq j}{\max}\sqrt{m_{{\scriptscriptstyle i}}}+
X^{-\xi_{{\scriptscriptstyle\omega}}}\bigg\}\\
&\ll_{j} X^{-1/2}Y^{2^{j-2}-1}(\log{2Y})^{j-1}+X^{-\xi_{{\scriptscriptstyle\omega}}}Y^{2^{j-2}-3/2}(\log{2Y})^{j-1}\,.
\end{split}
\end{equation}
The proof of Lemma $3$ is complete.
\end{proof}
\noindent
We need one final result before we can proceed to estimate the moments of $\widehat{\mathcal{E}}(x;\omega)$. In the lemma below we establish an asymptotic mean value estimate, an estimate that will turn out to be the leading term in the $j$-th moment estimate for the truncated part of the series expansion of $\widehat{\mathcal{E}}(x;\omega)$.       
\begin{lemma}
Let $\omega\in\mathbf{\Omega}$, and let $Y\geq1$. Then for $j\equiv0\,(2)$ we have
\begin{equation}\label{eq:2.46}
\begin{split}
(-1)^{\frac{j}{2}}\bigg(\frac{\sqrt{2}}{\pi}\bigg)^{j}&\sum_{e_{{\scriptscriptstyle1}},\ldots,e_{{\scriptscriptstyle j}}=\pm1}\bigg\{\prod_{i=1}^{j}e_{{\scriptscriptstyle i}}\bigg\}\frac{1}{X}\int\limits_{X}^{2X}\bigg\{\,\,\underset{\sum_{i=1}^{j}e_{{\scriptscriptstyle i}}\sqrt{m_{{\scriptscriptstyle i}}}=0}{\sum_{m_{{\scriptscriptstyle1}},\ldots,m_{{\scriptscriptstyle j}}\,\leq\,Y}}\prod_{i=1}^{j}\frac{r_{{\scriptscriptstyle 2}}(m_{{\scriptscriptstyle i}})}{m_{{\scriptscriptstyle i}}}\sin{\big(\pi\sqrt{m_{{\scriptscriptstyle i}}}\omega(x)\big)}\bigg\}\textit{d}x\\
&=\Bigg\{\frac{2^{2j}j!}{(j/2)!}+O_{j}\bigg(\Big|\log{\Big(\,\underset{X<x<2X}{\max}\omega(x)\Big)}\Big|^{-1}\bigg)\Bigg\}\mathscr{M}_{j}(X;\omega)+O_{j}\big(Y^{-1}\log{2Y}\big)\,.
\end{split}
\end{equation}
\end{lemma}
\begin{proof}
Fix an integer $j\equiv0\,(2)$. For $x>0$ and $e_{{\scriptscriptstyle1}},\ldots,e_{{\scriptscriptstyle j}}=\pm1$, we have
\begin{equation}\label{eq:2.47}
\begin{split}
&\underset{\sum_{i=1}^{j}e_{{\scriptscriptstyle i}}\sqrt{m_{{\scriptscriptstyle i}}}=0}{\sum_{m_{{\scriptscriptstyle1}},\ldots,m_{{\scriptscriptstyle j}}\,\leq\,Y}}\prod_{i=1}^{j}\frac{r_{{\scriptscriptstyle 2}}(m_{{\scriptscriptstyle i}})}{m_{{\scriptscriptstyle i}}}\sin{\big(\pi\sqrt{m_{{\scriptscriptstyle i}}}\omega(x)\big)}\\
&=\sum_{m_{1},\ldots,m_{j}\,\leq\,Y}\prod_{i=1}^{j}\frac{\mu^{2}(m_{i})}{m_{i}}\,\,\underset{\sum_{i=1}^{j}\textit{e}_{i}k_{i}\sqrt{m_{i}}=0}{\sum_{k_{1}\,\leq\,\sqrt{Y/m_{1}},\ldots,k_{j}\,\leq\,\sqrt{Y/m_{j}}}}\,\,\prod_{i=1}^{j}\frac{r_{{\scriptscriptstyle 2}}\big(m_{i}k^{2}_{i}\big)}{k^{2}_{i}}\sin{\big(\pi\sqrt{m_{{\scriptscriptstyle i}}}k_{i}\omega(x)\big)}\,.
\end{split}
\end{equation}
Now, the set $\{\sqrt{m}: m\text{ is square free}\}$ consists of $\mathbb{Q}-$linearly independent elements (this is Besicovitch's Lemma \cite{Besicovitch}). It follows that the relation $\sum_{i=1}^{j}\textit{e}_{i}k_{i}\sqrt{m_{i}}=0$ with $m_{1},\ldots,m_{j}$ square-free, is equivalent to a set of relations $\sum_{i\in\mathfrak{P}_{\ell}}\textit{e}_{i}k_{i}=0$ for  $1\leq\ell\leq s$, where $\biguplus_{\ell=1}^{s}\mathfrak{P}_{\ell}=\{1,,\ldots,j\}$ is a partition corresponding to a grouping of the $m_{i}$'s into distinct elements. Clearly, we have $|\mathfrak{P}_{\ell}|\geq2$ for each $\ell$. Denoting by $\mathcal{N}(j;\ell_{1},\ldots,\ell_{s})$ the number of partitions $\biguplus_{\ell=1}^{s}\mathfrak{P}_{\ell}=\{1,,\ldots,j\}$ with $\mathfrak{P}_{i}=\ell_{i}$, it follows that
\begin{equation}\label{eq:2.48}
\begin{split}
\sum_{e_{{\scriptscriptstyle1}},\ldots,e_{{\scriptscriptstyle j}}=\pm1}&\bigg\{\prod_{i=1}^{j}e_{{\scriptscriptstyle i}}\bigg\}\underset{\sum_{i=1}^{j}e_{{\scriptscriptstyle i}}\sqrt{m_{{\scriptscriptstyle i}}}=0}{\sum_{m_{{\scriptscriptstyle1}},\ldots,m_{{\scriptscriptstyle j}}\,\leq\,Y}}\prod_{i=1}^{j}\frac{r_{{\scriptscriptstyle 2}}(m_{{\scriptscriptstyle i}})}{m_{{\scriptscriptstyle i}}}\sin{\big(\pi\sqrt{m_{{\scriptscriptstyle i}}}\omega(x)\big)}\\
&=\sum_{s=1}^{j/2}\underset{\ell_{1},\,\ldots\,,\ell_{s}\geq2}{\sum_{\ell_{1}+\cdots+\ell_{s}=j}}\mathcal{N}(j;\ell_{1},\ldots,\ell_{s})\,\,\sideset{}{^\sharp}\sum_{m_{1},\ldots,m_{s}\leq Y}\prod_{i=1}^{s}\mathcal{P}\big(x;m_{i},\ell_{i},\sqrt{Y/m_{i}}\,\big)\\
&=\frac{j!}{2^{j/2}(j/2)!}\sideset{}{^\sharp}\sum_{m_{1},\ldots,m_{j/2}\leq Y}\prod_{i=1}^{j/2}\mathcal{P}\big(x;m_{i},2,\sqrt{Y/m_{i}}\,\big)\\
&\,\,\,\,\,+\sum_{s=1}^{j/2-1}\underset{\ell_{1},\,\ldots\,,\ell_{s}\geq2}{\sum_{\ell_{1}+\cdots+\ell_{s}=j}}\mathcal{N}(j;\ell_{1},\ldots,\ell_{s})\,\,\sideset{}{^\sharp}\sum_{m_{1},\ldots,m_{s}\leq Y}\prod_{i=1}^{s}\mathcal{P}\big(x;m_{i},\ell_{i},\sqrt{Y/m_{i}}\,\big)\,,
\end{split}
\end{equation}
where the symbol $\sharp$ indicates that the summation is over pairwise distinct integers $m_{1},\ldots,m_{s}$, and for $y\geq1$, $m\geq1$ and $\ell\geq2$, the term $\mathcal{P}(x;m,\ell,y)$ is defined by
\begin{equation}\label{eq:2.49}
\mathcal{P}(x;m,\ell,y)=\frac{\mu^{2}(m)}{m^{\ell}}\sum_{\textit{e}_{1},\ldots,\textit{e}_{\ell}=\pm1}\bigg\{\prod_{i=1}^{\ell}e_{{\scriptscriptstyle i}}\bigg\}\underset{\textit{e}_{1}k_{1}+\cdots+\textit{e}_{\ell}k_{\ell}=0}{\sum_{k_{1},\ldots,k_{\ell}\leq y}}\prod_{i=1}^{\ell}\frac{r_{{\scriptscriptstyle2}}\big(mk^{2}_{i}\big)}{k_{i}^{2}}\sin{\big(\pi\sqrt{m_{{\scriptscriptstyle i}}}k_{i}\omega(x)\big)}\,.
\end{equation}
We will show that the term $\sum^{\sharp}_{m_{1},\ldots,m_{j/2}\leq Y}\prod_{i=1}^{j/2}\mathcal{P}\big(x;m_{i},2,\sqrt{Y/m_{i}}\,\big)$ is dominant, while all other terms are negligible. In order to do so, we shall need to examine $\mathcal{P}(x;m,\ell,y)$ more closely.\\
First, we claim that for any $x>0$, any integer $m$ and any $y\geq1$, we have
\begin{equation}\label{eq:2.50}
\ell\geq2:\quad\big|\mathcal{P}(x;m,\ell,y)\big|\ll_{\ell}\omega^{\ell}(x)\frac{r^{\ell}_{{\scriptscriptstyle2}}(m)}{m^{\ell/2}}\,,
\end{equation}
To prove \eqref{eq:2.50} let $x>0$, let $m$ be an integer, $y\geq1$, $\ell\geq2$ an integer, and let $e_{{\scriptscriptstyle i}},\ldots,e_{{\scriptscriptstyle\ell}}=\pm1$.  We may assume that $m$ is square-free since otherwise $\mathcal{P}(x;m,\ell,y)$ vanishes. We have
%Since $\sin{(t)}/t$ is bounded, we have
%
%
\begin{equation}\label{eq:2.51}
\begin{split}
\underset{\textit{e}_{1}k_{1}+\cdots+\textit{e}_{\ell}k_{\ell}=0}{\sum_{k_{1},\ldots,k_{\ell}\leq y}}&\prod_{i=1}^{\ell}\frac{r_{{\scriptscriptstyle2}}\big(mk^{2}_{i}\big)}{k_{i}^{2}}\sin{\big(\pi\sqrt{m}k_{i}\omega(x)\big)}\\
&=\pi^{\ell}\omega^{\ell}(x)m^{\ell/2}\underset{\textit{e}_{1}k_{1}+\cdots+\textit{e}_{\ell}k_{\ell}=0}{\sum_{k_{1},\ldots,k_{\ell}\leq y}}\prod_{i=1}^{\ell}\frac{r_{{\scriptscriptstyle2}}\big(mk^{2}_{i}\big)}{k_{i}}\frac{\sin{\big(\pi\sqrt{m}k_{i}\omega(x)\big)}}{\pi\sqrt{m}k_{i}\omega(x)}\,.
\end{split}
\end{equation}
The relation $\textit{e}_{1}k_{1}+\cdots+\textit{e}_{\ell}k_{\ell}=0$ implies that $\textit{e}_{i}=-1$ for some $1\leq i\leq\ell$. Without loss of generality we may assume that $\textit{e}_{\ell}=-1$. Since $m$ is square-free we have $r_{{\scriptscriptstyle2}}(mk^{2})\leq r_{{\scriptscriptstyle2}}(m)d^{2}_{{\scriptscriptstyle2}}(k)$ where $d_{{\scriptscriptstyle2}}(\cdot)$ is the divisor function, and using the boundedness of $\sin{(t)}/t$ we find that
\begin{equation}\label{eq:2.52}
\begin{split}
\bigg|\underset{\textit{e}_{1}k_{1}+\cdots+\textit{e}_{\ell}k_{\ell}=0}{\sum_{k_{1},\ldots,k_{\ell}\leq y}}&\prod_{i=1}^{\ell}\frac{r_{{\scriptscriptstyle2}}\big(mk^{2}_{i}\big)}{k_{i}^{2}}\sin{\big(\pi\sqrt{m}k_{i}\omega(x)\big)}\bigg|\\
&\ll\omega^{\ell}(x)m^{\ell/2}r^{\ell}_{{\scriptscriptstyle2}}(m)\underset{1\leq\sum_{i=1}^{\ell-1}\textit{e}_{i}k_{i}\leq y}{\sum_{k_{1},\ldots,k_{\ell-1}\leq y}}\bigg\{\prod_{i=1}^{\ell-1}\frac{d^{2}_{{\scriptscriptstyle2}}\big(k_{i}\big)}{k_{i}}\bigg\}\frac{d^{2}_{{\scriptscriptstyle2}}\big(\sum_{i=1}^{\ell-1}\textit{e}_{i}k_{i}\big)}{\sum_{i=1}^{\ell-1}\textit{e}_{i}k_{i}}\\
&=\omega^{\ell}(x)m^{\ell/2}r^{\ell}_{{\scriptscriptstyle2}}(m)\sum_{n\leq y}\frac{d^{2}_{{\scriptscriptstyle2}}(n)}{n}\underset{\sum_{i=1}^{\ell-1}\textit{e}_{i}k_{i}=n}{\sum_{k_{1},\ldots,k_{\ell-1}\leq y}}\prod_{i=1}^{\ell-1}\frac{d^{2}_{{\scriptscriptstyle2}}\big(k_{i}\big)}{k_{i}}\,.
\end{split}
\end{equation}
Suppose first that $\ell\geq3$. Fix a positive integer $n$, and suppose that $k_{1},\ldots,k_{\ell-1}\leq y$ are positive integers satisfying $\sum_{i=1}^{\ell-1}\textit{e}_{i}k_{i}=n$. Let $k=\max_{1\leq i\leq \ell-1}k_{i}$, and without loss of generality we may assume that $k=k_{\ell-1}$. The relation $\sum_{i=1}^{\ell-1}\textit{e}_{i}k_{i}=n$ implies that $k_{\ell-1}\gg_{\ell}(n\prod_{i=1}^{\ell-2}k_{i})^{1/\ell}$, and since $d_{{\scriptscriptstyle2}}(b)\ll b^{1/2}$, it follows that $\prod_{i=1}^{\ell-1}d^{2}_{{\scriptscriptstyle2}}(k_{i})/k_{i}\ll_{\ell}n^{-1/2\ell}\prod_{i=1}^{\ell-2}d^{2}_{{\scriptscriptstyle2}}(k_{i})/k^{1+1/2\ell}_{i}$. It then follows from \eqref{eq:2.52} that
\begin{equation}\label{eq:2.53}
\begin{split}
\ell\geq3:\quad\bigg|\underset{\textit{e}_{1}k_{1}+\cdots+\textit{e}_{\ell}k_{\ell}=0}{\sum_{k_{1},\ldots,k_{\ell}\leq y}}\prod_{i=1}^{\ell}\frac{r_{{\scriptscriptstyle2}}\big(mk^{2}_{i}\big)}{k_{i}^{2}}\sin{\big(\pi\sqrt{m}k_{i}\omega(x)\big)}\bigg|&\ll\omega^{\ell}(x)m^{\ell/2}r^{\ell}_{{\scriptscriptstyle2}}(m)\bigg(\,\,\sum_{k\leq y}\frac{d^{2}_{{\scriptscriptstyle2}}(k)}{k^{1+1/2\ell}}\bigg)^{\ell-1}\\
&\ll_{\ell}\omega^{\ell}(x)m^{\ell/2}r^{\ell}_{{\scriptscriptstyle2}}(m)\,.
\end{split}
\end{equation}
In the case of $\ell=2$ it follows immediately from \eqref{eq:2.52} that
\begin{equation}\label{eq:2.54}
\begin{split}
\ell=2:\quad\bigg|\underset{k_{1}-k_{2}=0}{\sum_{k_{1}, k_{2}\leq y}}\prod_{i=1}^{2}\frac{r_{{\scriptscriptstyle2}}\big(mk^{2}_{i}\big)}{k_{i}^{2}}\sin{\big(\pi\sqrt{m}k_{i}\omega(x)\big)}\bigg|&\ll\omega^{2}(x)mr^{2}_{{\scriptscriptstyle2}}(m)\sum_{k\leq y}\frac{d^{4}_{{\scriptscriptstyle2}}(k)}{k^{2}}\\
&\ll\omega^{2}(x)mr^{2}_{{\scriptscriptstyle2}}(m)\,.
\end{split}
\end{equation}
It is now straightforward to deduce the bound \eqref{eq:2.50} from \eqref{eq:2.53} and \eqref{eq:2.54}. Observe that by \eqref{eq:2.50} we have 
\begin{equation}\label{eq:2.55}
\begin{split}
\ell\geq3:\quad\bigg|\sum_{m\leq Y}\mathcal{P}\big(x;m,\ell,\sqrt{Y/m}\,\big)\bigg|&\ll_{\ell}\omega^{\ell}(x)\sum_{m\leq Y}\frac{r^{\ell}_{{\scriptscriptstyle2}}(m)}{m^{\ell/2}}\\
&\ll_{\ell}\omega^{\ell}(x)\,.
\end{split}
\end{equation}
Next, we consider the case $\ell=2$ for which we shall need to evaluate $\sum_{m\leq Y}\mathcal{P}\big(x;m,2,\sqrt{Y/m}\,\big)$ precisely. For $m\leq Y$ we have
\begin{equation}\label{eq:2.56}
\begin{split}
\mathcal{P}\big(x;m,\ell,\sqrt{Y/m}\,\big)&=\frac{\mu^{2}(m)}{m^{2}}\sum_{\textit{e}_{1},\textit{e}_{2}=\pm1}\bigg\{\prod_{i=1}^{2}e_{{\scriptscriptstyle i}}\bigg\}\underset{\textit{e}_{1}k_{1}+\textit{e}_{2}k_{2}=0}{\sum_{k_{1},k_{2}\leq\sqrt{Y/m}}}\prod_{i=1}^{2}\frac{r_{{\scriptscriptstyle2}}\big(mk^{2}_{i}\big)}{k_{i}^{2}}\sin{\big(\pi\sqrt{m_{{\scriptscriptstyle i}}}k_{i}\omega(x)\big)}\\
&=-2\mu^{2}(m)\sum_{mk^{2}\leq Y}\frac{r^{2}_{{\scriptscriptstyle2}}\big(mk^{2}\big)}{\big(mk^{2}\big)^{2}}\sin^{2}{\big(\pi\sqrt{mk^{2}}\omega(x)\big)}
\end{split}
\end{equation}
Since every positive integer $n$ can be written uniquely as $n=mk^{2}$ with $m$ square-free, it follows from \eqref{eq:2.56} that
\begin{equation}\label{eq:2.57}
\sum_{m\leq Y}\mathcal{P}\big(x;m,2,\sqrt{Y/m}\,\big)=-2\sum_{n\leq Y}\frac{r^{2}_{{\scriptscriptstyle2}}(n)}{n^{2}}\sin^{2}{\big(\pi\sqrt{n}\omega(x)\big)}
\end{equation}
In order to evaluate the RHS of \eqref{eq:2.57} we shall need to have an estimate for $\sum_{n\leq y}r^{2}_{{\scriptscriptstyle2}}(n)$. For $y\geq1$ we have (see \cite{Ramanujan})
\begin{equation}\label{eq:2.58}
\sum_{n\leq y}r^{2}_{{\scriptscriptstyle2}}(n)=4y\log{y}+O(y)\,.
\end{equation}
%
%(noting that $1/\omega(x)\to\infty$)
Let $x>0$ be large. By \eqref{eq:2.58} and partial summation we obtain
\begin{equation}\label{eq:2.59}
\begin{split}
\sum_{n\leq Y}\frac{r^{2}_{{\scriptscriptstyle2}}(n)}{n^{2}}\sin^{2}{\big(\pi\sqrt{n}\omega(x)\big)}&=\sum_{n=1}^{\infty}\frac{r^{2}_{{\scriptscriptstyle2}}(n)}{n^{2}}\sin^{2}{\big(\pi\sqrt{n}\omega(x)\big)}+O(Y^{-1}\log{2Y})\\
&=\sum_{n\leq(1/2\pi\omega(x))^{2}}\frac{r^{2}_{{\scriptscriptstyle2}}(n)}{n^{2}}\sin^{2}{\big(\pi\sqrt{n}\omega(x)\big)}+O\big(\omega^{2}(x)|\log\omega(x)|+Y^{-1}\log{2Y}\big)\,.
\end{split}
\end{equation}
Using the series approximation $\sin{(t)}/t=1+O(t^{2})$ with $0<t\leq1/2$, we have for $n\leq(1/2\pi\omega(x))^{2}$
\begin{equation}\label{eq:2.60}
\begin{split}
\frac{r^{2}_{{\scriptscriptstyle2}}(n)}{n^{2}}\sin^{2}{\big(\pi\sqrt{n}\omega(x)\big)}&=(\pi\omega(x))^{2}\frac{r^{2}_{{\scriptscriptstyle2}}(n)}{n}\bigg(\frac{\sin{\big(\pi\sqrt{n}\omega(x)\big)}}{\pi\sqrt{n}\omega(x)}\bigg)^{2}\\
&=(\pi\omega(x))^{2}\frac{r^{2}_{{\scriptscriptstyle2}}(n)}{n}+O\big(\omega^{4}(x)r^{2}_{{\scriptscriptstyle2}}(n)\big)\,.
\end{split}
\end{equation}
It follows from \eqref{eq:3.44} and partial summation that
\begin{equation}\label{eq:2.61}
\begin{split}
\sum_{n\leq(1/2\pi\omega(x))^{2}}\frac{r^{2}_{{\scriptscriptstyle2}}(n)}{n^{2}}\sin^{2}{\big(\pi\sqrt{n}\omega(x)\big)}&=(\pi\omega(x))^{2}\sum_{n\leq(1/2\pi\omega(x))^{2}}\frac{r^{2}_{{\scriptscriptstyle2}}(n)}{n}+O\big(\omega^{2}(x)|\log\omega(x)|\big)\\
&=(\pi\omega(x))^{2}\Big\{8\big(\log{\omega(x)}\big)^{2}+O\big(|\log{\omega(x)}|\big)\Big\}+O\big(\omega^{2}(x)|\log\omega(x)|\big)\\
&=\big(2^{3/2}\pi\omega(x)\log{\omega(x)}\big)^{2}+O\big(\omega^{2}(x)|\log\omega(x)|\big)
\end{split}
\end{equation}
From \eqref{eq:2.57}, \eqref{eq:2.59} and \eqref{eq:2.61} we find that
\begin{equation}\label{eq:2.62}
\sum_{m\leq Y}\mathcal{P}\big(x;m,2,\sqrt{Y/m}\,\big)=-\big(4\pi\omega(x)\log{\omega(x)}\big)^{2}+O\big(\omega^{2}(x)|\log\omega(x)|+Y^{-1}\log{2Y}\big)\,.
\end{equation}
We are now ready to evaluate the RHS of \eqref{eq:2.48}. First we consider the terms $\sum^{\sharp}_{m_{1},\ldots,m_{s}\leq Y}\prod_{i=1}^{s}\mathcal{P}\big(x;m_{i},\ell_{i},\sqrt{Y/m_{i}}\,\big)$ with $s\leq j/2-1$ (in particular $j\geq4$) where $\ell_{1},\,\ldots\,,\ell_{s}\geq2$ are integers satisfying $\ell_{1}+\cdots+\ell_{s}=j$. Observe that since $s\leq j/2-1$ it must be the case that $\ell_{i}\geq3$ for some $i$. It follows from \eqref{eq:2.55} and \eqref{eq:2.62} that 
%Let $\varkappa$ be the number of $\ell_{i}$'s satisfying $\ell_{i}\geq3$. Observe that since $s\leq j/2-1$ it must be that $\varkappa\geq1$.
%
%
\begin{equation}\label{eq:2.63}
\begin{split}
s\leq j/2-1:\quad\bigg|\sideset{}{^\sharp}\sum_{m_{1},\ldots,m_{s}\leq Y}&\prod_{i=1}^{s}\mathcal{P}\big(x;m_{i},\ell_{i},\sqrt{Y/m_{i}}\,\big)\bigg|\\
&\ll_{j}\omega^{j}(x)(\log{\omega(x)})^{j-4}+Y^{-1}\log{2Y}\,.
\end{split}
\end{equation}
For the main term, appealing to \eqref{eq:2.50} with $\ell=2$ and to \eqref{eq:2.62} , we find that
\begin{equation}\label{eq:2.64}
\begin{split}
\sideset{}{^\sharp}\sum_{m_{1},\ldots,m_{j/2}\leq Y}&\prod_{i=1}^{j/2}\mathcal{P}\big(x;m_{i},2,\sqrt{Y/m_{i}}\,\big)\\
&=\bigg\{\sum_{m\leq Y}\mathcal{P}\big(x;m,2,\sqrt{Y/m}\,\big)\bigg\}^{j/2}+O_{j}\big(\omega^{j}(x)(\log{\omega(x)})^{j-2}+Y^{-1}\log{2Y}\big)\\
&=(-1)^{\frac{j}{2}}2^{2j}\pi^{j}\big(\omega(x)\log{\omega(x)}\big)^{j}+O_{j}\big(\omega^{j}(x)|\log{\omega(x)}|^{j-1}+Y^{-1}\log{2Y}\big)\,.
\end{split}
\end{equation}
Finally, inserting \eqref{eq:2.63} and \eqref{eq:2.64} into the RHS of \eqref{eq:2.48}, multiplying both sides of the equation by $(-1)^{j/2}(\sqrt{2}/\pi)^{j}$ and then integrating over the range $X<x<2X$ with $X>0$ large, we arrive at
\begin{equation}\label{eq:2.65}
\begin{split}
(-1)^{\frac{j}{2}}\bigg(\frac{\sqrt{2}}{\pi}\bigg)^{j}&\sum_{e_{{\scriptscriptstyle1}},\ldots,e_{{\scriptscriptstyle j}}=\pm1}\bigg\{\prod_{i=1}^{j}e_{{\scriptscriptstyle i}}\bigg\}\frac{1}{X}\int\limits_{X}^{2X}\bigg\{\,\,\underset{\sum_{i=1}^{j}e_{{\scriptscriptstyle i}}\sqrt{m_{{\scriptscriptstyle i}}}=0}{\sum_{m_{{\scriptscriptstyle1}},\ldots,m_{{\scriptscriptstyle j}}\,\leq\,Y}}\prod_{i=1}^{j}\frac{r_{{\scriptscriptstyle 2}}(m_{{\scriptscriptstyle i}})}{m_{{\scriptscriptstyle i}}}\sin{\big(\pi\sqrt{m_{{\scriptscriptstyle i}}}\omega(x)\big)}\bigg\}\textit{d}x\\
&=\Bigg\{\frac{2^{2j}j!}{(j/2)!}+O_{j}\bigg(\Big|\log{\Big(\,\underset{X<x<2X}{\max}\omega(x)\Big)}\Big|^{-1}\bigg)\Bigg\}\mathscr{M}_{j}(X;\omega)+O_{j}\big(Y^{-1}\log{2Y}\big)\,.
\end{split}
\end{equation}
This concludes the proof.
\end{proof}
\subsection{Moment estimates}
The following are the main results of $\S2$.
\begin{prop}
Let $\omega\in\mathbf{\Omega}$. Then we have
\begin{equation}\label{eq:2.66}
\begin{split}
&j\equiv1\,(2):\quad\lim\limits_{X\to\infty}\bigg\{\frac{1}{X}\int\limits_{X}^{2X}\widehat{\mathcal{E}}^{j}(x;\omega)\textit{d}x\bigg\}\mathscr{M}^{-j/2}_{2}(X;\omega)=0\\
&j\equiv0\,(2):\quad\lim\limits_{X\to\infty}\bigg\{\frac{1}{X}\int\limits_{X}^{2X}\widehat{\mathcal{E}}^{j}(x;\omega)\textit{d}x\bigg\}\mathscr{M}^{-1}_{j}(X;\omega)=\frac{2^{2j}j!}{(j/2)!}\,.
\end{split}
\end{equation}
\end{prop}
\begin{proof}
Fix a positive integer $j$. Let $X>0$ be large, and let $X<x<2X$. Since $x^{-2}|\Xi_{\psi}(x;\omega)|\ll1$, it follows from Proposition $1$ that
\begin{equation}\label{eq:2.67}
\begin{split}
\widehat{\mathcal{E}}^{j}(x;\omega)&=\Bigg\{\frac{2^{3/2}}{\pi}\sum_{m\,\leq\,X^{2}}\frac{r_{{\scriptscriptstyle 2}}(m)}{m}\sin{\big(\pi\sqrt{m}\omega(x)\big)}\sin{\big(\pi\sqrt{m}(2x+\omega(x))\big)}\Bigg\}^{j}\\
&\,\,\,\,\,+O_{j,\epsilon}\Big((\log{X})^{j-1}X^{-2}|\Xi_{\psi}(x;\omega)|+X^{-1+\epsilon}\Big)\,.
\end{split}
\end{equation}
Multiplying both sides of \eqref{eq:2.67} by $1/X$, integrating over the prescribed range and using Lemma $1$ together with Cauchy–Schwarz inequalit, we find that 
\begin{equation}\label{eq:2.68}
\begin{split}
\frac{1}{X}\int\limits_{X}^{2X}\widehat{\mathcal{E}}^{j}(x;\omega)\textit{d}x&=\frac{1}{X}\int\limits_{X}^{2X}\bigg\{\frac{2^{3/2}}{\pi}\sum_{m\,\leq\,X^{2}}\frac{r_{{\scriptscriptstyle 2}}(m)}{m}\sin{\big(\pi\sqrt{m}\omega(x)\big)}\sin{\big(\pi\sqrt{m}(2x+\omega(x))\big)}\bigg\}^{j}\textit{d}x\\
&\,\,\,\,\,+O_{j,\epsilon}(X^{-1+\epsilon})\,.
\end{split}
\end{equation}
Now, let $Y=Y_{{\scriptscriptstyle j,\omega}}\geq1$ be a parameter that depends both on $j$ and $\omega$ that satisfies $Y<X^{2}$, and is to be determined later on. We split the sum over $m$ under the integral sign appearing on the RHS of \eqref{eq:2.68} into two ranges $m\leq Y$ and $Y<m\leq X^{2}$, obtaining
\begin{equation}\label{eq:2.69}
\begin{split}
\frac{1}{X}\int\limits_{X}^{2X}&\bigg\{\frac{2^{3/2}}{\pi}\sum_{m\,\leq\,X^{2}}\frac{r_{{\scriptscriptstyle 2}}(m)}{m}\sin{\big(\pi\sqrt{m}\omega(x)\big)}\sin{\big(\pi\sqrt{m}(2x+\omega(x))\big)}\bigg\}^{j}\textit{d}x\\
&=\frac{1}{X}\int\limits_{X}^{2X}\bigg\{\frac{2^{3/2}}{\pi}\sum_{m\,\leq\,Y}\frac{r_{{\scriptscriptstyle 2}}(m)}{m}\sin{\big(\pi\sqrt{m}\omega(x)\big)}\sin{\big(\pi\sqrt{m}(2x+\omega(x))\big)}\bigg\}^{j}\textit{d}x+E_{j,Y}(X)\,,
\end{split}
\end{equation}
where by Lemma $2$ and Cauchy–Schwarz inequality, we find that $E_{j,Y}(X)$ satisfies the bound
\begin{equation}\label{eq:2.70}
\begin{split}
|E_{j,Y}(X)|&\ll_{j}(\log{X})^{j-1}\Bigg\{\frac{1}{X}\int\limits_{X}^{2X}\bigg|\sum_{Y<m\,\leq\,X^{2}}\frac{r_{{\scriptscriptstyle 2}}(m)}{m}\sin{\big(\pi\sqrt{m}\omega(x)\big)}\sin{\big(\pi\sqrt{m}(2x+\omega(x))\big)}\bigg|^{2}\textit{d}x\Bigg\}^{1/2}\\
&\ll Y^{-1/8}(\log{X})^{j}\,.
\end{split}
\end{equation}
It follows from \eqref{eq:2.68}, \eqref{eq:2.69} and \eqref{eq:2.70} that
\begin{equation}\label{eq:2.71}
\begin{split}
\frac{1}{X}\int\limits_{X}^{2X}\widehat{\mathcal{E}}^{j}(x;\omega)\textit{d}x&=\frac{1}{X}\int\limits_{X}^{2X}\bigg\{\frac{2^{3/2}}{\pi}\sum_{m\,\leq\,Y}\frac{r_{{\scriptscriptstyle 2}}(m)}{m}\sin{\big(\pi\sqrt{m}\omega(x)\big)}\sin{\big(\pi\sqrt{m}(2x+\omega(x))\big)}\bigg\}^{j}\textit{d}x\\
&\,\,\,\,\,+O\big(Y^{-1/8}(\log{X})^{j}\big)\,.
\end{split}
\end{equation}
We proceed to estimate the integral appearing on the RHS of \eqref{eq:2.71}, and we do so in accordance with the parity of $j$.\\
\textbf{\textnormal{\textbf{case (i):}}} $j\equiv1\,(2)$. Expanding out the product of the $\sin{}$ functions in which $\omega$ appears in the argument, we have the following decomposition 
\begin{equation}\label{eq:2.72}
\begin{split}
\frac{1}{X}\int\limits_{X}^{2X}\bigg\{\frac{2^{3/2}}{\pi}&\sum_{m\,\leq\,Y}\frac{r_{{\scriptscriptstyle 2}}(m)}{m}\sin{\big(\pi\sqrt{m}\omega(x)\big)}\sin{\big(\pi\sqrt{m}(2x+\omega(x))\big)}\bigg\}^{j}\textit{d}x\\
&=(-1)^{\big[\frac{j}{2}\big]}\bigg(\frac{\sqrt{2}}{\pi}\bigg)^{j}\sum_{e_{{\scriptscriptstyle1}},\ldots,e_{{\scriptscriptstyle j}}=\pm1}\bigg\{\prod_{i=1}^{j}e_{{\scriptscriptstyle i}}\bigg\}\text{\Large{W}}^{\text{(odd)},\neq}_{(e_{{\scriptscriptstyle1}},\ldots,e_{{\scriptscriptstyle j}})\,}(Y,X)\,,
\end{split}
\end{equation}
where for $e_{{\scriptscriptstyle1}},\ldots,e_{{\scriptscriptstyle j}}=\pm1$, we define
\begin{equation}\label{eq:2.73}
\begin{split}
\text{\Large{W}}^{\text{(odd)},\neq}_{(e_{{\scriptscriptstyle1}},\ldots,e_{{\scriptscriptstyle j}})\,}(Y,X)&=\frac{1}{X}\int\limits_{X}^{2X}\sum_{m_{{\scriptscriptstyle1}},\ldots,m_{{\scriptscriptstyle j}}\,\leq\,Y}\bigg\{\prod_{i=1}^{j}\frac{r_{{\scriptscriptstyle 2}}(m_{{\scriptscriptstyle i}})}{m_{{\scriptscriptstyle i}}}\sin{\big(\pi\sqrt{m_{{\scriptscriptstyle i}}}\omega(x)\big)}\bigg\}\sin{\bigg(\pi\Big(\sum_{i=1}^{j}e_{{\scriptscriptstyle i}}\sqrt{m_{{\scriptscriptstyle i}}}\Big)(2x+\omega(x))\bigg)}\textit{d}x\\
&=\frac{1}{X}\int\limits_{X}^{2X}\underset{\sum_{i=1}^{j}e_{{\scriptscriptstyle i}}\sqrt{m_{{\scriptscriptstyle i}}}\neq0}{\sum_{m_{{\scriptscriptstyle1}},\ldots,m_{{\scriptscriptstyle j}}\,\leq\,Y}}\bigg\{\prod_{i=1}^{j}\frac{r_{{\scriptscriptstyle 2}}(m_{{\scriptscriptstyle i}})}{m_{{\scriptscriptstyle i}}}\sin{\big(\pi\sqrt{m_{{\scriptscriptstyle i}}}\omega(x)\big)}\bigg\}\sin{\bigg(\pi\Big(\sum_{i=1}^{j}e_{{\scriptscriptstyle i}}\sqrt{m_{{\scriptscriptstyle i}}}\Big)(2x+\omega(x))\bigg)}\textit{d}x\,.
\end{split}
\end{equation}
We may bound each of the summands $\text{\Large{W}}^{\text{(odd)},\neq}_{(e_{{\scriptscriptstyle1}},\ldots,e_{{\scriptscriptstyle j}})\,}(Y,X)$ by appealing to Lemma $3$, obtaining
%\text{\Large{H}}_{(e_{{\scriptscriptstyle1}},e_{{\scriptscriptstyle2}})\,}(Y,X)
%
\begin{equation}\label{eq:2.74}
\begin{split}
\Bigg|\frac{1}{X}\int\limits_{X}^{2X}\bigg\{\frac{2^{3/2}}{\pi}&\sum_{m\,\leq\,Y}\frac{r_{{\scriptscriptstyle 2}}(m)}{m}\sin{\big(\pi\sqrt{m}\omega(x)\big)}\sin{\big(\pi\sqrt{m}(2x+\omega(x))\big)}\bigg\}^{j}\textit{d}x\Bigg|\\
&\ll_{j}\sum_{e_{{\scriptscriptstyle1}},\ldots,e_{{\scriptscriptstyle j}}=\pm1}\big|\text{\Large{W}}^{\text{(odd)},\neq}_{(e_{{\scriptscriptstyle1}},\ldots,e_{{\scriptscriptstyle j}})\,}(Y,X)\big|\\
&\ll_{j}\left\{
        \begin{array}{ll}
            X^{-1/2}\log{2Y}+X^{-\xi_{{\scriptscriptstyle\omega}}}& ;\,j=1
            \\\\
            \big(X^{-1/2}Y^{1/2}+X^{-\xi_{{\scriptscriptstyle\omega}}}\big)Y^{2^{j-2}-3/2}(\log{2Y})^{j-1}& ;\,j>1\text{ and }j\equiv1\,(2)\,.
        \end{array}
    \right.
\end{split}
\end{equation}
Inserting these bounds into the RHS of \eqref{eq:2.71} we find that
\begin{equation}\label{eq:2.75}
\begin{split}
\Bigg|\frac{1}{X}\int\limits_{X}^{2X}\widehat{\mathcal{E}}^{j}(x;\omega)\textit{d}x\Bigg|&\ll_{j} Y^{-1/8}(\log{X})^{j}\\
&\,\,\,\,\,+\left\{
        \begin{array}{ll}
            X^{-1/2}\log{2Y}+X^{-\xi_{{\scriptscriptstyle\omega}}}& ;\,j=1
            \\\\
            \big(X^{-1/2}Y^{1/2}+X^{-\xi_{{\scriptscriptstyle\omega}}}\big)Y^{2^{j-2}-3/2}(\log{2Y})^{j-1}& ;\,j>1\text{ and }j\equiv1\,(2)\,.
        \end{array}
    \right.
\end{split}
\end{equation}
Our aim is to choose a suitable $Y$ so that the RHS of \eqref{eq:2.75} decays like smoe power of $X$. This is easily achieved by choosing $Y$ to be a sufficiently small power of $X$. We set $\theta=\min\big\{\eta, \xi_{{\scriptscriptstyle\omega}}, \frac{1}{2}\big\}$, and choose
\begin{equation}\label{eq:2.76}
\begin{split}
Y=\left\{
        \begin{array}{ll}
            X^{\theta}& ;\,j=1
            \\\\
            X^{\theta/2^{3}(2^{j-2}-3/2)}& ;\,j>1\text{ and }j\equiv1\,(2)\,.
        \end{array}
    \right.
\end{split}
\end{equation}
With this choice we find that
%\text{\Large{H}}_{(e_{{\scriptscriptstyle1}},e_{{\scriptscriptstyle2}})\,}(Y,X)
%
\begin{equation}\label{eq:2.77}
j\equiv1\,(2):\quad\Bigg|\frac{1}{X}\int\limits_{X}^{2X}\widehat{\mathcal{E}}^{j}(x;\omega)\textit{d}x\Bigg|\ll_{j}X^{-\theta/16}\,.
\end{equation}
Now, according to condition \textbf{(4\textnormal{\textbf{a}})} we have $\tau(\omega)=\inf\{\alpha>0: \mathscr{M}_{2}(X;\omega)\gg X^{-\alpha}\}=0$. In particular, $\mathscr{M}_{2}(X;\omega)\gg X^{-\theta/16j}$. Multiplying \eqref{eq:2.77} by $\mathscr{M}^{-j/2}_{2}(X;\omega)$ we arrive at
\begin{equation}\label{eq:2.78}
j\equiv1\,(2):\quad\Bigg|\frac{1}{X}\int\limits_{X}^{2X}\widehat{\mathcal{E}}^{j}(x;\omega)\textit{d}x\Bigg|\mathscr{M}^{-j/2}_{2}(X;\omega)\ll_{j}X^{-\theta/32}\,.
\end{equation}
Letting $X\to\infty$ in \eqref{eq:2.78} we obtain the proof of \eqref{eq:2.88} for $j\equiv1\,(2)$.\\
\textbf{\textnormal{\textbf{case (ii):}}} $j\equiv0\,(2)$. We may claerly assume that $j\geq2$. Expanding out the product of the $\sin{}$ functions in which $\omega$ appears in the argument, we have the following decomposition
\begin{equation}\label{eq:2.79}
\begin{split}
\begin{split}
\frac{1}{X}\int\limits_{X}^{2X}\bigg\{\frac{2^{3/2}}{\pi}&\sum_{m\,\leq\,Y}\frac{r_{{\scriptscriptstyle 2}}(m)}{m}\sin{\big(\pi\sqrt{m}\omega(x)\big)}\sin{\big(\pi\sqrt{m}(2x+\omega(x))\big)}\bigg\}^{j}\textit{d}x\\
&=(-1)^{\frac{j}{2}}\bigg(\frac{\sqrt{2}}{\pi}\bigg)^{j}\sum_{e_{{\scriptscriptstyle1}},\ldots,e_{{\scriptscriptstyle j}}=\pm1}\bigg\{\prod_{i=1}^{j}e_{{\scriptscriptstyle i}}\bigg\}\bigg\{\text{\Large{W}}^{\text{(even)},=}_{(e_{{\scriptscriptstyle1}},\ldots,e_{{\scriptscriptstyle j}})\,}(Y,X)+\text{\Large{W}}^{\text{(even)},\neq}_{(e_{{\scriptscriptstyle1}},\ldots,e_{{\scriptscriptstyle j}})\,}(Y,X)\bigg\}\,,
\end{split}
\end{split}
\end{equation}
where for $e_{{\scriptscriptstyle1}},\ldots,e_{{\scriptscriptstyle j}}=\pm1$, we define
\begin{equation}\label{eq:2.80}
\begin{split}
\text{\Large{W}}^{\text{(even)},=}_{(e_{{\scriptscriptstyle1}},\ldots,e_{{\scriptscriptstyle j}})\,}(Y,X)&=\frac{1}{X}\int\limits_{X}^{2X}\underset{\sum_{i=1}^{j}e_{{\scriptscriptstyle i}}\sqrt{m_{{\scriptscriptstyle i}}}=0}{\sum_{m_{{\scriptscriptstyle1}},\ldots,m_{{\scriptscriptstyle j}}\,\leq\,Y}}\bigg\{\prod_{i=1}^{j}\frac{r_{{\scriptscriptstyle 2}}(m_{{\scriptscriptstyle i}})}{m_{{\scriptscriptstyle i}}}\sin{\big(\pi\sqrt{m_{{\scriptscriptstyle i}}}\omega(x)\big)}\bigg\}\cos{\bigg(\pi\Big(\sum_{i=1}^{j}e_{{\scriptscriptstyle i}}\sqrt{m_{{\scriptscriptstyle i}}}\Big)(2x+\omega(x))\bigg)}\textit{d}x\\
&=\frac{1}{X}\int\limits_{X}^{2X}\bigg\{\,\,\underset{\sum_{i=1}^{j}e_{{\scriptscriptstyle i}}\sqrt{m_{{\scriptscriptstyle i}}}=0}{\sum_{m_{{\scriptscriptstyle1}},\ldots,m_{{\scriptscriptstyle j}}\,\leq\,Y}}\prod_{i=1}^{j}\frac{r_{{\scriptscriptstyle 2}}(m_{{\scriptscriptstyle i}})}{m_{{\scriptscriptstyle i}}}\sin{\big(\pi\sqrt{m_{{\scriptscriptstyle i}}}\omega(x)\big)}\bigg\}\textit{d}x\,,
\end{split}
\end{equation}
and 
\begin{equation}\label{eq:2.81}
\begin{split}
\text{\Large{W}}^{\text{(even)},\neq}_{(e_{{\scriptscriptstyle1}},\ldots,e_{{\scriptscriptstyle j}})\,}(Y,X)&=\frac{1}{X}\int\limits_{X}^{2X}\underset{\sum_{i=1}^{j}e_{{\scriptscriptstyle i}}\sqrt{m_{{\scriptscriptstyle i}}}\neq0}{\sum_{m_{{\scriptscriptstyle1}},\ldots,m_{{\scriptscriptstyle j}}\,\leq\,Y}}\bigg\{\prod_{i=1}^{j}\frac{r_{{\scriptscriptstyle 2}}(m_{{\scriptscriptstyle i}})}{m_{{\scriptscriptstyle i}}}\sin{\big(\pi\sqrt{m_{{\scriptscriptstyle i}}}\omega(x)\big)}\bigg\}\cos{\bigg(\pi\Big(\sum_{i=1}^{j}e_{{\scriptscriptstyle i}}\sqrt{m_{{\scriptscriptstyle i}}}\Big)(2x+\omega(x))\bigg)}\textit{d}x\,.
\end{split}
\end{equation}
We bound each summand $\text{\Large{W}}^{\text{(even)},\neq}_{(e_{{\scriptscriptstyle1}},\ldots,e_{{\scriptscriptstyle j}})\,}(Y,X)$ using Lemma $3$, obtaining
\begin{equation}\label{eq:2.82}
\big|\text{\Large{W}}^{\text{(even)},\neq}_{(e_{{\scriptscriptstyle1}},\ldots,e_{{\scriptscriptstyle j}})\,}(Y,X)\big|\ll_{j}\big(X^{-1/2}Y^{1/2}+X^{-\xi_{{\scriptscriptstyle\omega}}}\big)Y^{2^{j-2}-3/2}(\log{2Y})^{j-1}\,.
\end{equation}
According to Lemma $4$ we have
\begin{equation}\label{eq:2.83}
\begin{split}
\begin{split}
(-1)^{\frac{j}{2}}\bigg(\frac{\sqrt{2}}{\pi}\bigg)^{j}&\sum_{e_{{\scriptscriptstyle1}},\ldots,e_{{\scriptscriptstyle j}}=\pm1}\bigg\{\prod_{i=1}^{j}e_{{\scriptscriptstyle i}}\bigg\}\bigg\{\text{\Large{W}}^{\text{(even)},=}_{(e_{{\scriptscriptstyle1}},\ldots,e_{{\scriptscriptstyle j}})\,}(Y,X)+\text{\Large{W}}^{\text{(even)},\neq}_{(e_{{\scriptscriptstyle1}},\ldots,e_{{\scriptscriptstyle j}})\,}(Y,X)\bigg\}\\
&=\Bigg\{\frac{2^{2j}j!}{(j/2)!}+O_{j}\bigg(\Big|\log{\Big(\,\underset{X<x<2X}{\max}\omega(x)\Big)}\Big|^{-1}\bigg)\Bigg\}\mathscr{M}_{j}(X;\omega)+O_{j}\big(Y^{-1}\log{2Y}\big)\,.
\end{split}
\end{split}
\end{equation}
It follows from \eqref{eq:2.71}, \eqref{eq:2.79}, \eqref{eq:2.82} and \eqref{eq:2.83} that
\begin{equation}\label{eq:2.84}
\begin{split}
\frac{1}{X}\int\limits_{X}^{2X}\widehat{\mathcal{E}}^{j}(x;\omega)\textit{d}x&=\Bigg\{\frac{2^{2j}j!}{(j/2)!}+O_{j}\bigg(\Big|\log{\Big(\,\underset{X<x<2X}{\max}\omega(x)\Big)}\Big|^{-1}\bigg)\Bigg\}\mathscr{M}_{j}(X;\omega)\\
&\,\,\,\,\,+O\Big(Y^{-1/8}(\log{X})^{j}+\big(X^{-1/2}Y^{1/2}+X^{-\xi_{{\scriptscriptstyle\omega}}}\big)Y^{2^{j-2}-3/2}(\log{2Y})^{j-1}\Big)\,.
\end{split}
\end{equation}
We are now ready to make a choice for $Y$. We set $\theta=\min\big\{\eta, \xi_{{\scriptscriptstyle\omega}}, \frac{1}{2}\big\}$, and choose
\begin{equation}\label{eq:2.85}
\begin{split}
Y=\left\{
        \begin{array}{ll}
            X^{\theta}& ;\,j=2
            \\\\
            X^{\theta/2^{3}(2^{j-2}-3/2)}& ;\,j>2\text{ and }j\equiv0\,(2)\,.
        \end{array}
    \right.
\end{split}
\end{equation}
With this choice we find that the error term on the RHS of \eqref{eq:2.84} satisfies the bound
\begin{equation}\label{eq:2.86}
Y^{-1/8}(\log{X})^{j}+\big(X^{-1/2}Y^{1/2}+X^{-\xi_{{\scriptscriptstyle\omega}}}\big)Y^{2^{j-2}-3/2}(\log{2Y})^{j-1}\ll_{j}X^{-\theta/16}\,.
\end{equation}
Now, it follows from H$\ddot{\text{o}}$lder's inequality that $\mathscr{M}_{j}(X;\omega)\geq\mathscr{M}^{j/2}_{2}(X;\omega)$, and according to assumption \textbf{(4\textnormal{\textbf{a}})} we have $\tau(\omega)=\inf\{\alpha>0: \mathscr{M}_{2}(X;\omega)\gg X^{-\alpha}\}=0$. In particular, $\mathscr{M}_{j}(X;\omega)\gg X^{-\theta/32}$. Dividing \eqref{eq:2.84} by $\mathscr{M}_{j}(X;\omega)$ throughout, it then follows from \eqref{eq:2.86} that
\begin{equation}\label{eq:2.87}
\begin{split}
j\equiv0\,(2):\quad\bigg\{\frac{1}{X}\int\limits_{X}^{2X}\widehat{\mathcal{E}}^{j}(x;\omega)\textit{d}x\bigg\}\mathscr{M}^{-1}_{j}(X;\omega)&=\frac{2^{2j}j!}{(j/2)!}+O_{j}\bigg(X^{-\theta/16}+\Big|\log{\Big(\,\underset{X<x<2X}{\max}\omega(x)\Big)}\Big|^{-1}\bigg)\,.
\end{split}
\end{equation}
According to the assumptions on $\omega$, the error term on the RHS of \eqref{eq:2.87} tends to zero as $X\to\infty$, and therefore \eqref{eq:2.66} is proved in the case $j\equiv0\,(2)$. This concludes the proof Proposition $2$.
\end{proof}
\noindent
As an immediate corollary we obtain the following.
\begin{prop}
Let $\omega\in\mathbf{\Omega}$, and recall that the variance of $\widehat{\mathcal{E}}(x;\omega)$ is defined to be $\sigma^{2}(X;\omega)=X^{-1}\int_{X}^{2X}\widehat{\mathcal{E}}^{2}(x;\omega)\textit{d}x$. Then we have
\begin{equation}\label{eq:2.88}
\begin{split}
j\equiv1\,(2)\quad:&\lim\limits_{X\to\infty}\frac{1}{X}\int\limits_{X}^{2X}\bigg(\frac{\widehat{\mathcal{E}}(x;\omega)}{\sigma(X;\omega)}\bigg)^{j}\textit{d}x=0\\
j\equiv0\,(2)\quad:&\lim\limits_{X\to\infty}\frac{1}{X}\int\limits_{X}^{2X}\bigg(\frac{\widehat{\mathcal{E}}(x;\omega)}{\sigma(X;\omega)}\bigg)^{j}\textit{d}x=\frac{j!}{2^{j/2}(j/2)!}\mathscr{L}_{j}(\omega)\,.
\end{split}
\end{equation}
\begin{proof}
For $j\equiv1\,(2)$ we have
\begin{equation}\label{eq:2.89}
\begin{split}
\frac{1}{X}\int\limits_{X}^{2X}\bigg(\frac{\widehat{\mathcal{E}}(x;\omega)}{\sigma(X;\omega)}\bigg)^{j}\textit{d}x=\bigg\{\frac{1}{X}\int\limits_{X}^{2X}\widehat{\mathcal{E}}^{j}(x;\omega)\textit{d}x\bigg\}\mathscr{M}^{-j/2}_{2}(X;\omega)\Bigg(\bigg\{\frac{1}{X}\int\limits_{X}^{2X}\widehat{\mathcal{E}}^{2}(x;\omega)\textit{d}x\bigg\}\mathscr{M}^{-1}_{2}(X;\omega)\Bigg)^{-j/2}\,,
\end{split}
\end{equation}
while for $j\equiv0\,(2)$ we have 
\begin{equation}\label{eq:2.90}
\begin{split}
\frac{1}{X}\int\limits_{X}^{2X}\bigg(&\frac{\widehat{\mathcal{E}}(x;\omega)}{\sigma(X;\omega)}\bigg)^{j}\textit{d}x\\
&=\bigg\{\frac{1}{X}\int\limits_{X}^{2X}\widehat{\mathcal{E}}^{j}(x;\omega)\textit{d}x\bigg\}\mathscr{M}^{-1}_{j}(X;\omega)\Bigg(\bigg\{\frac{1}{X}\int\limits_{X}^{2X}\widehat{\mathcal{E}}^{2}(x;\omega)\textit{d}x\bigg\}\mathscr{M}^{-1}_{2}(X;\omega)\Bigg)^{-j/2}\frac{\mathscr{M}_{j}(X;\omega)}{\mathscr{M}^{j/2}_{2}(X;\omega)}\,.
\end{split}
\end{equation}
Letting $X\to\infty$, we obtain \eqref{eq:2.88} from \eqref{eq:2.66} (together with conditions \textbf{(4\textnormal{\textbf{b}})} in the case where $j\equiv0\,(2)$). This concludes the proof.
\end{proof}
\end{prop}
\section{Proof of the main results and corollaries}
\noindent
Before presenting the proofs of the main theorem and corollaries, we need to establish two lemmas.
\begin{lemma}
Let $\omega\in\mathbf{\Omega}$. Then there exists a unique (Radon) probability measure $\mu_{{\scriptscriptstyle\omega}}$ on the real line that satisfies
%$\mu_{{\scriptscriptstyle\omega}}$ is moment-determinate, i.e., it is uniquely characterised by its moments, and these are given by
%has finite moments of any order $\int_{-\infty}^{\infty}|\alpha|^{k}\textit{d}\mu_{{\scriptscriptstyle h}}(\alpha)<\infty$, $k=1,2,\ldots$. and is moment-determinatethe unique probability measure satisfying
\begin{equation}\label{eq:5.1}
\int\limits_{-\infty}^{\infty}\alpha^{j}\textit{d}\mu_{{\scriptscriptstyle\omega}}(\alpha)=\left\{
        \begin{array}{ll}
           \frac{j!}{2^{j/2}(j/2)!}\mathscr{L}_{j}(\omega) & ;\,j\equiv0\,(2)
            \\\\
           0 & ;\, j\equiv1\,(2)\,.
        \end{array}
    \right.
\end{equation}
\end{lemma}
\begin{proof}
Let $\omega\in\mathbf{\Omega}$, and define the sequence $(\mathfrak{m}_{\omega,j})_{j=0}^{\infty}$ by $\mathfrak{m}_{\omega,j}=\frac{j!}{2^{j/2}(j/2)!}\mathscr{L}_{j}(\omega)$ for $j\equiv0\,(2)$, and $\mathfrak{m}_{\omega,j}=0$ for $j\equiv1\,(2)$. We will show that the sequence just defined is positive
semidefinite, that is, for any $\xi_{1},\ldots,\xi_{n}\in\mathbb{R}$ we have $\sum_{j,\ell=0}^{n}\mathfrak{m}_{\omega,j+\ell}\xi_{j}\xi_{\ell}\geq0$. By the well known theorem of Hamburger (\cite{Hamburger}), we may deduce the existence of a (Radon) measure that satisfies \eqref{eq:5.1} (the case $j=0$ implies that any such measure has to be a probability measure). Moreover, such a measure is unique, since according to condition \textbf{(4\textnormal{\textbf{c}})} the sequence of even moments $(\mathfrak{m}_{\omega,j})_{j\equiv0\,(2)}^{\infty}$ satisfies Carleman’s criterion (\cite{Carleman}). So, let $\xi_{1},\ldots,\xi_{n}\in\mathbb{R}$ be given. We have
\begin{equation}\label{eq:5.2}
\begin{split}
\sum^{n}_{j,\ell=0}\mathfrak{m}_{\omega,j+\ell}\xi_{j}\xi_{\ell}&=\sum^{n}_{j\equiv\ell(2)}\frac{(j+\ell)!}{2^{(j+\ell)/2}((j+\ell)/2)!}\Bigg\{\lim_{X\to\infty}\frac{\mathscr{M}_{j+\ell}(X;\omega)}{\mathscr{M}^{(j+\ell)/2}_{2}(X;\omega)}\Bigg\}\xi_{j}\xi_{\ell}\\
&=\lim_{X\to\infty}\frac{1}{X}\int\limits_{X}^{2X}\Bigg\{\sum^{n}_{j,\ell=0}\mathfrak{g}_{j+\ell}\hat{\xi}_{\omega,j}(X;x)\hat{\xi}_{\omega,\ell}(X;x)\Bigg\}\textit{d}x\,,
\end{split}
\end{equation}
where for $X,x>0$ and $j=1,\ldots,n$ we set $\hat{\xi}_{\omega,j}(X;x)=(\omega(x)\log{\omega(x)})^{j}\mathscr{M}^{-j/2}_{2}(X;\omega)$, and the sequence $(\mathfrak{g}_{j})_{j=0}^{\infty}$ is given by $\mathfrak{g}_{j}=\frac{j!}{2^{j/2}(j/2)!}$ for $j\equiv0\,(2)$, and $\mathfrak{g}_{j}=0$ for $j\equiv1\,(2)$. Now, since $(\mathfrak{g}_{j})_{j=0}^{\infty}$ is semidefinite, it follows that
\begin{equation}\label{eq:5.3}
\sum^{n}_{j,\ell=0}\mathfrak{g}_{j+\ell}\hat{\xi}_{\omega,j}(X;x)\hat{\xi}_{\omega,\ell}(X;x)\geq0\,.
\end{equation}
Multiplying \eqref{eq:5.3} by $1/X$, integrating over the range $X<x<2X$ and then letting $X\to\infty$, we conclude from \eqref{eq:5.2} that $\sum^{n}_{j,\ell=0}\mathfrak{m}_{\omega,j+\ell}\xi_{j}\xi_{\ell}\geq0$ as required.
\end{proof}
\begin{lemma}
Let $\omega\in\mathbf{\Omega}$. Then from any sequence $(X_{n})_{n=1}^{\infty}$, $0<X_{n}\to\infty$, we may extract a sub-sequence $(Y_{k}:=X_{n_{k}})_{k=1}^{\infty}$, $Y_{k}\to\infty$, such that for any bounded continuous function $\mathcal{F}$ we have
\begin{equation}\label{eq:5.4}
\lim\limits_{k\to\infty}\frac{1}{Y_{k}}\int\limits_{ Y_{k}}^{2Y_{k}}\mathcal{F}\bigg(\frac{\widehat{\mathcal{E}}(x;\omega)}{\sigma(Y_{k};\omega)}\bigg)\textit{d}x=\int\limits_{-\infty}^{\infty}\mathcal{F}(\alpha)\textit{d}\mu_{{\scriptscriptstyle\omega}}(\alpha)\,,
\end{equation}
where $\mu_{{\scriptscriptstyle\omega}}$ is the measure from Lemma $5$.
\end{lemma}
\begin{proof}
Let $\omega\in\mathbf{\Omega}$ be given, and let $(X_{n})_{n=1}^{\infty}$ be sequence satisfying $0<X_{n}\to\infty$. According to Proposition $3$ we may find $X_{\omega}>0$ such that $X^{-1}\int_{X}^{2X}(\widehat{\mathcal{E}}(x;\omega)/\sigma(X;\omega))^{2}\textit{d}x\leq2\mathscr{L}_{2}(\omega)$ for all $X\geq X_{\omega}$. Consider the sequence $\big(\mu_{{\scriptscriptstyle X_{n},\omega}}: X_{n}\geq X_{\omega}\big)$ of distributions, where recall that $\mu_{{\scriptscriptstyle X,\omega}}$ is defined for an interval $\mathcal{A}$ on the real line by
\begin{equation*}
\int_{\mathcal{A}}\textit{d}\mu_{{\scriptscriptstyle X,\omega}}=\frac{1}{X}\textit{meas}\big\{X<x<2X:\widehat{\mathcal{E}}(x;\omega)/\sigma(X;\omega)\in\mathcal{A}\big\}\,,
\end{equation*}
where $\textit{meas}$ is the Lebesgue measure. We claim that this sequence is tight, that is, for any $\epsilon>0$ there exists $R=R_{\epsilon}>0$, such that $\int_{|x|>R}\textit{d}\mu_{{\scriptscriptstyle X_{n},\omega}}(x)<\epsilon$ for all $X_{n}\geq X_{\omega}$. Indeed, let $\epsilon>0$ be given and set $R_{\epsilon}=2\sqrt{\mathscr{L}_{2}(\omega)/\epsilon}$. We then have
%$\int_{\mathcal{A}}\textit{d}\mu_{{\scriptscriptstyle X,\omega}}(x)=X^{-1}\textit{meas}\big\{X<x<2X:\widehat{\mathcal{E}}(x;\omega)/\sigma(X;\omega)\in\mathcal{A}\big\}$ of distributions, where $\textit{meas}$ is the Lebesgue measure.\\
%
%
\begin{equation}\label{eq:5.6}
\begin{split}
X_{n}\geq X_{\omega}:\quad\int_{|x|>R_{\epsilon}}\textit{d}\mu_{{\scriptscriptstyle X_{n},\omega}}(x)&=\frac{1}{X_{n}}\int\limits_{X_{n}}^{2X_{n}}\mathds{1}_{\big\{|\widehat{\mathcal{E}}(x;\omega)/\sigma(X_{n};\omega)|>R_{\epsilon}\big\}}\textit{d}x\\
&\leq\frac{1}{R_{\epsilon}^{2}}\frac{1}{X_{n}}\int\limits_{X_{n}}^{2X_{n}}\bigg(\frac{\widehat{\mathcal{E}}(x;\omega)}{\sigma(X_{n};\omega)}\bigg)^{2}\textit{d}x<\epsilon\,,
\end{split}
\end{equation}
which proves the claim.\\
The tightness of $\big(\mu_{{\scriptscriptstyle X_{n},\omega}}: X_{n}\geq X_{\omega}\big)$ allows us to appeal to the the Helly Compactness theorem, by which we may deduce the existence of a sub-sequence $(Y_{k}:=X_{n_{k}}\geq X_{\omega})_{k=1}^{\infty}$ satisfying $Y_{k}\to\infty$, and the existence of a limiting distribution, which we shall denote by $\upsilon$, such that $\mu_{{\scriptscriptstyle Y_{k},\omega}}\to\upsilon$ weakly, as $k\to\infty$. In other words, for any bounded continuous function $\mathcal{F}$ we have
\begin{equation}\label{eq:5.6}
\lim\limits_{k\to\infty}\frac{1}{Y_{k}}\int\limits_{ Y_{k}}^{2Y_{k}}\mathcal{F}\bigg(\frac{\widehat{\mathcal{E}}(x;\omega)}{\sigma(Y_{k};\omega)}\bigg)\textit{d}x=\int\limits_{-\infty}^{\infty}\mathcal{F}(\alpha)\textit{d}\upsilon(\alpha)\,.
\end{equation}
It remains to sow that $\upsilon=\mu_{{\scriptscriptstyle\omega}}$, where $\mu_{{\scriptscriptstyle\omega}}$ is the measure from Lemma $5$. First, observe that $\upsilon$ has finite moments of any order. Indeed, it suffices to show this for the even moments. To that end, let $\psi\in C^{\infty}_{0}(\mathbb{R})$ satisfy $0\leq\psi(y)\leq1$, and $\psi(y)=1$ for $|y|\leq1$. It then follows from \eqref{eq:5.6}, Proposition $3$ and Fatou's lemma, that\label{eq:5.7}
\begin{equation}
\begin{split}
\int\limits_{-\infty}^{\infty}\alpha^{2j}\textit{d}\upsilon(\alpha)\leq\liminf_{\ell\to\infty}\int\limits_{-\infty}^{\infty}\psi(\alpha/\ell)\alpha^{2j}\textit{d}\upsilon(\alpha)&=\liminf_{\ell\to\infty}\lim\limits_{k\to\infty}\frac{1}{Y_{k}}\int\limits_{ Y_{k}}^{2Y_{k}}\psi\bigg(\frac{\widehat{\mathcal{E}}(x;\omega)}{\sigma(Y_{k};\omega)\ell}\bigg)\bigg(\frac{\widehat{\mathcal{E}}(x;\omega)}{\sigma(Y_{k};\omega)}\bigg)^{2j}\textit{d}x\\
&\leq\lim\limits_{k\to\infty}\frac{1}{Y_{k}}\int\limits_{ Y_{k}}^{2Y_{k}}\bigg(\frac{\widehat{\mathcal{E}}(x;\omega)}{\sigma(Y_{k};\omega)}\bigg)^{2j}\textit{d}x<\infty\,.
\end{split}
\end{equation}
It now follows from \eqref{eq:5.6}, Proposition $3$ and Lebesgue's dominated convergence theorem, that for any non-negative integer $j$
\begin{equation}\label{eq:5.8}
\begin{split}
\int\limits_{-\infty}^{\infty}\alpha^{j}\textit{d}\upsilon(\alpha)=\lim_{\ell\to\infty}\int\limits_{-\infty}^{\infty}\psi(\alpha/\ell)\alpha^{j}\textit{d}\upsilon(\alpha)&=\lim_{\ell\to\infty}\lim\limits_{k\to\infty}\frac{1}{Y_{k}}\int\limits_{ Y_{k}}^{2Y_{k}}\psi\bigg(\frac{\widehat{\mathcal{E}}(x;\omega)}{\sigma(Y_{k};\omega)\ell}\bigg)\bigg(\frac{\widehat{\mathcal{E}}(x;\omega)}{\sigma(Y_{k};\omega)}\bigg)^{j}\textit{d}x\\
&=\lim\limits_{\ell\to\infty}\Bigg\{\lim\limits_{k\to\infty}\frac{1}{Y_{k}}\int\limits_{Y_{k}}^{2Y_{k}}\Bigg(\frac{\widehat{\mathcal{E}}(x;\omega)}{\sigma(Y_{k};\omega)}\bigg)^{j}\textit{d}x+O_{j}\bigg(\frac{1}{\ell}\bigg)\Bigg\}\\
&=\left\{
        \begin{array}{ll}
           \frac{j!}{2^{j/2}(j/2)!}\mathscr{L}_{j}(\omega) & ;\,j\equiv0\,(2)
            \\\\
           0 & ;\, j\equiv1\,(2)\,.
        \end{array}
    \right.
\end{split}
\end{equation}
Thus, $\upsilon$ has the same moments as $\mu_{{\scriptscriptstyle\omega}}$. By uniqueness of $\mu_{{\scriptscriptstyle\omega}}$, it must be that $\upsilon=\mu_{{\scriptscriptstyle\omega}}$. This concludes the proof. 
\end{proof}
\noindent
We are now ready to present the proof of Theorem $1$.
\begin{proof}(Theorem 1). Fix a gap width function $\omega\in\mathbf{\Omega}$, and let $\mu_{{\scriptscriptstyle\omega}}$ be the measure in Lemma $5$. We claim that for any bounded continuous function $\mathcal{F}$, it holds
\begin{equation}\label{eq:5.9}
\lim\limits_{X\to\infty}\frac{1}{X}\int\limits_{ X}^{2X}\mathcal{F}\bigg(\frac{\widehat{\mathcal{E}}(x;\omega)}{\sigma(X;\omega)}\bigg)\textit{d}x=\int\limits_{-\infty}^{\infty}\mathcal{F}(\alpha)\textit{d}\mu_{{\scriptscriptstyle\omega}}(\alpha)\,.
\end{equation}
To prove the claim, we argue by contradiction. Suppose that there exists a bounded continuous function $\mathcal{K}$ such that \eqref{eq:5.9} does not hold. We may then find a sequence $(X_{n})_{n=1}^{\infty}$, $0<X_{n}\to\infty$, such that the limit
\begin{equation}\label{eq:5.10}
I(\mathcal{K}):=\lim\limits_{n\to\infty}\frac{1}{X_{n}}\int\limits_{ X}^{2X_{n}}\mathcal{K}\bigg(\frac{\widehat{\mathcal{E}}(x;\omega)}{\sigma(X_{n};\omega)}\bigg)\textit{d}x
\end{equation}
exists, and satisfies $I(\mathcal{F})\neq\int_{-\infty}^{\infty}\mathcal{F}(\alpha)\textit{d}\mu_{{\scriptscriptstyle\omega}}(\alpha)$. Now, according to Lemma $6$ we may extract a sub-sequence $(Y_{k}:=X_{n_{k}})_{k=1}^{\infty}$, $Y_{k}\to\infty$, such that for any bounded continuous function $\mathcal{F}$ we have
\begin{equation}\label{eq:5.11}
\lim\limits_{k\to\infty}\frac{1}{Y_{k}}\int\limits_{ Y_{k}}^{2Y_{k}}\mathcal{F}\bigg(\frac{\widehat{\mathcal{E}}(x;\omega)}{\sigma(Y_{k};\omega)}\bigg)\textit{d}x=\int\limits_{-\infty}^{\infty}\mathcal{F}(\alpha)\textit{d}\mu_{{\scriptscriptstyle\omega}}(\alpha)\,.
\end{equation}
Specializing \eqref{eq:5.11} to $\mathcal{F}=\mathcal{K}$, we obtain the desired contradiction
\begin{equation}\label{eq:5.12}
\begin{split}
\int\limits_{-\infty}^{\infty}\mathcal{K}(\alpha)\textit{d}\mu_{{\scriptscriptstyle\omega}}(\alpha)\neq I(\mathcal{K}):=\lim\limits_{n\to\infty}\frac{1}{X_{n}}\int\limits_{ X}^{2X_{n}}\mathcal{K}\bigg(\frac{\widehat{\mathcal{E}}(x;\omega)}{\sigma(X_{n};\omega)}\bigg)\textit{d}x&=\lim\limits_{k\to\infty}\frac{1}{Y_{k}}\int\limits_{ Y_{k}}^{2Y_{k}}\mathcal{K}\bigg(\frac{\widehat{\mathcal{E}}(x;\omega)}{\sigma(Y_{k};\omega)}\bigg)\textit{d}x\\
&=\int\limits_{-\infty}^{\infty}\mathcal{K}(\alpha)\textit{d}\mu_{{\scriptscriptstyle\omega}}(\alpha)\,.
\end{split}
\end{equation}
%bounded (piecewise)-continuous function $\mathcal{F}$ whose points of discontinuity are points of continuity of $\mu_{{\scriptscriptstyle\omega}}$
%
Hence \eqref{eq:5.9} holds for all bounded continuous functions $\mathcal{F}$. The extension of \eqref{eq:5.9} to include bounded (piecewise)-continuous function $\mathcal{F}$ whose points of discontinuity are points of continuity of $\mu_{{\scriptscriptstyle\omega}}$ is now straightforward, and so \eqref{eq:1.1} in Theorem $1$ is proved.\\
Item \textnormal{\textbf{(A)}} in Theorem $1$ follows from Lemma $5$, while item \textnormal{\textbf{(B)}} follows from Proposition $3$ and Lebesgue's dominated convergence theorem (in the exact same manner as \eqref{eq:5.8} is proved with $\alpha^{j}$ replaced by $\mathcal{F}(\alpha)$ which by assumption has polynomial growth). This concludes the proof of Theorem $1$.
\end{proof}
\noindent
The proof of corollaries \textbf{(2A)} and \textbf{(2B)} follow immediately from the following lemma. 
\begin{lemma}
Let $\varphi_{{\scriptscriptstyle\ell}}(t)=|p_{{\scriptscriptstyle\ell}}(\exp{(2\pi it)})|^{2}$, and $\lambda_{{\scriptscriptstyle\ell}}\in\mathbb{R}$ with $\ell=1,\ldots,n$ be given, where $p_{{\scriptscriptstyle\ell}}$ are polynomials with no roots on the unit circle. Fix an integer $A>1$, and let $\omega_{{\scriptscriptstyle\times}}(x):=\big\{\prod_{\ell=1}^{n}\varphi_{{\scriptscriptstyle\ell}}\big(\lambda_{{\scriptscriptstyle\ell}}(\log{x})^{A}\big)\big\}(\log{x})^{-A}$ and $\omega_{{\scriptscriptstyle+}}(x):=\big\{\sum_{\ell=1}^{n}\varphi_{{\scriptscriptstyle\ell}}\big(\lambda_{{\scriptscriptstyle\ell}}(\log{x})^{A}\big)\big\}(\log{x})^{-A}$. Then the following holds.\\
\textnormal{\textbf{(i)}} $\omega_{{\scriptscriptstyle\times}}, \omega_{{\scriptscriptstyle+}}\in\mathbf{\Omega}$.\\
\textnormal{\textbf{(ii)}} Suppose that $\lambda_{{\scriptscriptstyle1}},\ldots,\lambda_{{\scriptscriptstyle n}}$ are linearly independent over $\mathbb{Z}$. Let $\Phi_{{\scriptscriptstyle\times}}(\mathbf{t})=\Phi_{{\scriptscriptstyle\times}}((t_{{\scriptscriptstyle1}},\ldots, t_{{\scriptscriptstyle n}}))=\prod_{\ell=1}^{n}\varphi_{\ell}(t_{{\scriptscriptstyle\ell}})$ and $\Theta_{{\scriptscriptstyle+}}(\mathbf{t})=\Theta_{{\scriptscriptstyle+}}((t_{{\scriptscriptstyle1}},\ldots, t_{{\scriptscriptstyle n}}))=\sum_{\ell=1}^{n}\varphi_{\ell}(t_{{\scriptscriptstyle\ell}})$. Then
\begin{equation}\label{eq:3.13}
\begin{split}
&\lim_{X\to\infty}\frac{\mathscr{M}_{j}(X;\omega_{{\scriptscriptstyle\times}})}{\mathscr{M}^{j/2}_{2}(X;\omega_{{\scriptscriptstyle\times}})}=\bigg(\frac{\Vert\Phi_{{\scriptscriptstyle\times}}\Vert_{{\scriptscriptstyle j}}}{\Vert\Phi_{{\scriptscriptstyle\times}}\Vert_{{\scriptscriptstyle2}}}\bigg)^{j}\\
&\lim_{X\to\infty}\frac{\mathscr{M}_{j}(X;\omega_{{\scriptscriptstyle+}})}{\mathscr{M}^{j/2}_{2}(X;\omega_{{\scriptscriptstyle+}})}=\bigg(\frac{\Vert\Theta_{{\scriptscriptstyle+}}\Vert_{{\scriptscriptstyle j}}}{\Vert\Theta_{{\scriptscriptstyle+}}\Vert_{{\scriptscriptstyle2}}}\bigg)^{j}\,,
\end{split}
\end{equation}
where $\Vert\Phi_{{\scriptscriptstyle\times}}\Vert_{{\scriptscriptstyle j}}=\Big(\int_{\mathbf{t}\in[0,1)^{n}}\Phi^{j}_{{\scriptscriptstyle\times}}(\mathbf{t})\textit{d}\mathbf{t}\Big)^{1/j}$ and $\Vert\Theta_{{\scriptscriptstyle+}}\Vert_{{\scriptscriptstyle j}}=\Big(\int_{\mathbf{t}\in[0,1)^{n}}\Theta^{j}_{{\scriptscriptstyle+}}(\mathbf{t})\textit{d}\mathbf{t}\Big)^{1/j}$ denote the $j$-norm of $\Phi_{{\scriptscriptstyle\times}}$ and $\Theta_{{\scriptscriptstyle+}}$ respectively.
\end{lemma}
\begin{proof}
We shall prove the lemma only for $\omega_{{\scriptscriptstyle\times}}\in\mathbf{\Omega}$, the case of $\omega_{{\scriptscriptstyle+}}$ being identical. We begin with item \textnormal{\textbf{(i)}}. Clearly, $\omega_{{\scriptscriptstyle\times}}$ satisfies condition \textbf{(1)} and condition \textbf{(2)} in the definition of $\mathbf{\Omega}$ is satisfied. Letting $\varphi_{{\scriptscriptstyle\ell}}(t)=|p_{{\scriptscriptstyle\ell}}(\exp{(2\pi it)})|^{2}=\sum^{d_{{\scriptscriptstyle\ell}}}_{m=-d_{{\scriptscriptstyle\ell}}}a_{{\scriptscriptstyle m,\ell}}\exp{(2\pi im_{{\scriptscriptstyle}}t)}$, with $d_{{\scriptscriptstyle\ell}}$ a non-negative integer and where the coefficients satisfy the relation $a_{{\scriptscriptstyle-m,\ell}}=\Bar{a}_{{\scriptscriptstyle m,\ell}}$, we see that $\prod_{\ell=1}^{n}\varphi_{{\scriptscriptstyle\ell}}(\lambda_{{\scriptscriptstyle\ell}}t)$ can be expressed as an almost periodic (real) trigonometric polynomial, namely, denoting by $\mathfrak{F}$ the (finite) set $\mathfrak{F}=\{\sum^{n}_{\ell=1}m_{{\scriptscriptstyle\ell}}\lambda_{{\scriptscriptstyle\ell}}:|m_{{\scriptscriptstyle\ell}}|\leq d_{{\scriptscriptstyle\ell}}\}$, we have
%
%m_{{\scriptscriptstyle1}}\lambda_{{\scriptscriptstyle1}}+\cdots+m_{{\scriptscriptstyle n}}\lambda_{{\scriptscriptstyle n}}
\begin{equation}\label{eq:3.14}
\prod_{\ell=1}^{n}\varphi_{{\scriptscriptstyle\ell}}(\lambda_{{\scriptscriptstyle\ell}}t)=\sum_{\mathfrak{f}\in\mathfrak{F}}\mathfrak{a}_{{\scriptscriptstyle\mathfrak{f}}}\exp{\big(2\pi i\mathfrak{f}t\big)}\quad;\quad\mathfrak{a}_{{\scriptscriptstyle\mathfrak{f}}}=\underset{\sum^{n}_{\ell=1}m_{{\scriptscriptstyle\ell}}\lambda_{{\scriptscriptstyle\ell}}=\mathfrak{f}}{\sum_{|m_{{\scriptscriptstyle1}}|\leq d_{{\scriptscriptstyle1}},\ldots,|m_{{\scriptscriptstyle n}}|\leq d_{{\scriptscriptstyle n}}}}a_{m_{{\scriptscriptstyle1}},1}\cdots a_{m_{{\scriptscriptstyle n}},n}\,,
\end{equation}
where $f\in\mathfrak{F}$ iff $-f\in\mathfrak{F}$, and $\mathfrak{a}_{{\scriptscriptstyle-\mathfrak{f}}}=\Bar{\mathfrak{a}}_{{\scriptscriptstyle\mathfrak{f}}}$. For the derivatives of $\omega_{{\scriptscriptstyle\times}}$ we obtain
\begin{equation}
\omega^{{\scriptscriptstyle(1)}}_{{\scriptscriptstyle\times}}(x)=\frac{A}{x(\log{x})^{A+1}}\sum_{\mathfrak{f}\in\mathfrak{F}}\big(2\pi i\mathfrak{f}(\log{x})^{A}-1\big)\mathfrak{a}_{{\scriptscriptstyle\mathfrak{f}}}\exp{\big(2\pi i\mathfrak{f}(\log{x})^{A}\big)}\,,
\end{equation}
and 
\begin{equation}
\begin{split}
&\omega^{{\scriptscriptstyle(2)}}_{{\scriptscriptstyle\times}}(x)\\
&=-\frac{A}{x^{2}(\log{x})^{A+2}}\sum_{\mathfrak{f}\in\mathfrak{F}}\Big((A+1+\log{x})\big(2\pi i\mathfrak{f}(\log{x})^{A}-1\big)+A(2\pi\mathfrak{f})^{2}(\log{x})^{2A}\Big)\mathfrak{a}_{{\scriptscriptstyle\mathfrak{f}}}\exp{\big(2\pi i\mathfrak{f}(\log{x})^{A}\big)}\,.
\end{split}
\end{equation}
Let us define two functions, $\mathcal{H}(z)$ and $\mathcal{Q}(z)$ with $z\in\mathbb{C}$ a complex variable, by
\begin{equation}
\mathcal{H}(z)=\sum_{\mathfrak{f}\in\mathfrak{F}}\big(2\pi i\mathfrak{f}z^{A}-1\big)\mathfrak{a}_{{\scriptscriptstyle\mathfrak{f}}}\exp{\big(2\pi i\mathfrak{f}z^{A}\big)}\,,
\end{equation}
and
\begin{equation}
\mathcal{Q}(z)=\sum_{\mathfrak{f}\in\mathfrak{F}}\big((A+1+z)\big(2\pi i\mathfrak{f}z^{A}-1\big)+A(2\pi\mathfrak{f})^{2}z^{2A}\big)\mathfrak{a}_{{\scriptscriptstyle\mathfrak{f}}}\exp{\big(2\pi i\mathfrak{f}z^{A}\big)}\,,
\end{equation}
Then on recalling that $\mathcal{U}^{\omega_{\times}}_{{\scriptscriptstyle X}}=\{X\leq x\leq2X:\omega^{{\scriptscriptstyle(1)}}_{{\scriptscriptstyle\times}}(x)=0\}$ and $\mathcal{V}^{\omega_{\times}}_{{\scriptscriptstyle X}}:=\{X\leq x\leq2X:\omega^{{\scriptscriptstyle(2)}}_{{\scriptscriptstyle\times}}(x)=0\}$, it follows that
\begin{equation}\label{eq:3.19}
\begin{split}
&|\mathcal{U}^{\omega_{\times}}_{{\scriptscriptstyle X}}|=\big|\big\{x\in\mathbb{R}:\mathcal{H}(x)=0\text{ and }\log{X}\leq x\leq\log{2X}\big\}\big|\\
&|\mathcal{V}^{\omega_{\times}}_{{\scriptscriptstyle X}}|=\big|\big\{x\in\mathbb{R}:\mathcal{Q}(x)=0\text{ and }\log{X}\leq x\leq\log{2X}\big\}\big|\,.
\end{split}
\end{equation}
Now, $\mathcal{H}(z)$ and $\mathcal{Q}(z)$ are entire functions of $z$ of order $A$, and are both non-zero since $\mathcal{H}(0)=-\prod_{\ell=1}^{n}\varphi_{{\scriptscriptstyle\ell}}(0)\neq0$ and $\mathcal{H}(0)=-(A+1)\prod_{\ell=1}^{n}\varphi_{{\scriptscriptstyle\ell}}(0)\neq0$. An application of Jensen's formula then gives us an upper bound on the number of zeros of $\mathcal{H}(z)$ and $\mathcal{Q}(z)$ in a circle $|z|\leq R$ for large $R>0$
\begin{equation}\label{eq:3.20}
\begin{split}
&\big|\big\{z\in\mathbb{C}:\mathcal{H}(z)=0\text{ and }|z|\leq R\big\}\big|\ll R^{A}\\
&\big|\big\{z\in\mathbb{C}:\mathcal{Q}(z)=0\text{ and }|z|\leq R\big\}\big|\ll R^{A}\,,
\end{split}
\end{equation}
where the implied constant depends only on the frequencies $\mathfrak{f}$ and the coefficients $\mathfrak{a}_{{\scriptscriptstyle\mathfrak{f}}}$. Now, let $X>0$ be large. It then follows from \eqref{eq:3.19}, \eqref{eq:3.20} that 
\begin{equation}
\begin{split}
&|\mathcal{U}^{\omega_{\times}}_{{\scriptscriptstyle X}}|\leq\big|\big\{z\in\mathbb{C}:\mathcal{H}(z)=0\text{ and }|z|\leq\log{2X}\big\}\big|\ll(\log{X})^{A}\\
&|\mathcal{V}^{\omega_{\times}}_{{\scriptscriptstyle X}}|\leq\big|\big\{z\in\mathbb{C}:\mathcal{Q}(z)=0\text{ and }|z|\leq\log{2X}\big\}\big|\ll(\log{X})^{A}\,.
\end{split}
\end{equation}
We conclude that $|\mathcal{U}^{\omega_{{\scriptscriptstyle\times}}}_{{\scriptscriptstyle X}}|\max_{X\leq x\leq2X}\omega_{{\scriptscriptstyle\times}}(x)\ll1$, and that $|\mathcal{V}^{\omega_{{\scriptscriptstyle\times}}}_{{\scriptscriptstyle X}}|\ll(\log{X})^{A}$. Hence, conditions \textbf{(3\textnormal{\textbf{a}})} and \textbf{(3\textnormal{\textbf{b}})} are satisfied.\\
Now, let $j\equiv0\,(2)$ be a positive integer. First, observe that
\begin{equation}\label{eq:3.22}
0<m:=\prod^{n}_{\ell=1}\underset{t\in[0,1)}{\min}\varphi_{{\scriptscriptstyle\ell}}(t)\leq\prod_{\ell=1}^{n}\varphi_{{\scriptscriptstyle\ell}}\big(\lambda_{{\scriptscriptstyle\ell}}(\log{x})^{A}\big)\leq\prod^{n}_{\ell=1}\underset{t\in[0,1)}{\max}\varphi_{{\scriptscriptstyle\ell}}(t):=M<\infty\,.
\end{equation}
In what follows, the implied constants in the $O$-terms are allowed to depend on $m,M$ and $j$. We  have 
\begin{equation}\label{eq:3.23}
\begin{split}
\mathscr{M}_{j}(X;\omega_{{\scriptscriptstyle\times}})&=\frac{1}{X}\int\limits_{X}^{2X}\big(\omega_{{\scriptscriptstyle\times}}(x)\log{\omega_{{\scriptscriptstyle\times}}(x)}\big)^{j}\textit{d}x\\
&=\bigg(A\frac{\log{\log{X}}}{(\log{X})^{A}}\bigg)^{j}\Bigg\{\frac{1}{X}\int\limits_{X}^{2X}\bigg\{\prod_{\ell=1}^{n}\varphi_{{\scriptscriptstyle\ell}}\big(\lambda_{{\scriptscriptstyle\ell}}(\log{x})^{A}\big)\bigg\}^{j}\textit{d}x+O\bigg(\frac{1}{\log{\log{X}}}\bigg)\Bigg\}\,.
\end{split}
\end{equation}
From \eqref{eq:3.22} and \eqref{eq:3.23} it follows that $\tau(\omega_{{\scriptscriptstyle\times}}):=\inf\{\alpha>0: \mathscr{M}_{2}(X;\omega_{{\scriptscriptstyle\times}})\gg X^{-\alpha}\}=0$, hence condition \textbf{(4\textnormal{\textbf{a}})} is satisfied. Now, by \eqref{eq:3.14} and the first mean value theorem for integrals, we have
%
%
%
%=\sum_{\mathfrak{f}_{{\scriptscriptstyle1}},\ldots,\mathfrak{f}_{{\scriptscriptstyle j}}\in\mathfrak{F}}\bigg(\prod_{\ell=1}^{j}\mathfrak{a}_{{\scriptscriptstyle\mathfrak{f}}_{\ell}}\bigg)\frac{1}{X}\int\limits_{X}^{2X}\exp{\bigg(2\pi i\bigg\{\sum_{\ell=1}^{j}\mathfrak{f}_{{\scriptscriptstyle\ell}}\bigg\}(\log{x})^{A}\bigg)}\textit{d}x
\begin{equation}\label{eq:3.24}
\begin{split}
&\frac{1}{X}\int\limits_{X}^{2X}\bigg\{\prod_{\ell=1}^{n}\varphi_{{\scriptscriptstyle\ell}}\big(\lambda_{{\scriptscriptstyle\ell}}(\log{x})^{A}\big)\bigg\}^{j}\textit{d}x\\
&=\underset{\sum_{\ell=1}^{j}\mathfrak{f}_{{\scriptscriptstyle\ell}}=0}{\sum_{\mathfrak{f}_{{\scriptscriptstyle1}},\ldots,\mathfrak{f}_{{\scriptscriptstyle j}}\in\mathfrak{F}}}\bigg(\prod_{\ell=1}^{j}\mathfrak{a}_{{\scriptscriptstyle\mathfrak{f}}_{\ell}}\bigg)+\frac{1}{2\pi iA}\,\,\underset{\sum_{\ell=1}^{j}\mathfrak{f}_{{\scriptscriptstyle\ell}}\neq0}{\sum_{\mathfrak{f}_{{\scriptscriptstyle1}},\ldots,\mathfrak{f}_{{\scriptscriptstyle j}}\in\mathfrak{F}}}\Bigg(\frac{\prod_{\ell=1}^{j}\mathfrak{a}_{{\scriptscriptstyle\mathfrak{f}}_{\ell}}}{\sum_{\ell=1}^{j}\mathfrak{f}_{{\scriptscriptstyle\ell}}}\Bigg)\frac{1}{X}\int\limits_{X}^{2X}\frac{x}{(\log{x})^{A-1}}\textit{d}\Bigg(\exp{\bigg(2\pi i\bigg\{\sum_{\ell=1}^{j}\mathfrak{f}_{{\scriptscriptstyle\ell}}\bigg\}(\log{x})^{A}\bigg)}\Bigg)\\
&=\underset{\sum_{\ell=1}^{j}\mathfrak{f}_{{\scriptscriptstyle\ell}}=0}{\sum_{\mathfrak{f}_{{\scriptscriptstyle1}},\ldots,\mathfrak{f}_{{\scriptscriptstyle j}}\in\mathfrak{F}}}\bigg(\prod_{\ell=1}^{j}\mathfrak{a}_{{\scriptscriptstyle\mathfrak{f}}_{\ell}}\bigg)+O\Bigg(\frac{1}{(\log{x})^{A-1}}\underset{\sum_{\ell=1}^{j}\mathfrak{f}_{{\scriptscriptstyle\ell}}\neq0}{\sum_{\mathfrak{f}_{{\scriptscriptstyle1}},\ldots,\mathfrak{f}_{{\scriptscriptstyle j}}\in\mathfrak{F}}}\Bigg|\frac{\prod_{\ell=1}^{j}\mathfrak{a}_{{\scriptscriptstyle\mathfrak{f}}_{\ell}}}{\sum_{\ell=1}^{j}\mathfrak{f}_{{\scriptscriptstyle\ell}}}\Bigg|\Bigg)\,.
\end{split}
\end{equation}
Since $A>1$, it follows from \eqref{eq:3.23} and \eqref{eq:3.24} that $\mathscr{L}_{j}(\omega)$ exists, and is given by
\begin{equation}\label{eq:3.25}
\mathscr{L}_{j}(\omega)=\lim\limits_{X\to\infty}\frac{\mathscr{M}_{j}(X;\omega_{{\scriptscriptstyle\times}})}{\mathscr{M}^{j/2}_{2}(X;\omega_{{\scriptscriptstyle\times}})}=\Bigg\{\,\,\underset{\sum_{\ell=1}^{j}\mathfrak{f}_{{\scriptscriptstyle\ell}}=0}{\sum_{\mathfrak{f}_{{\scriptscriptstyle1}},\ldots,\mathfrak{f}_{{\scriptscriptstyle j}}\in\mathfrak{F}}}\bigg(\prod_{\ell=1}^{j}\mathfrak{a}_{{\scriptscriptstyle\mathfrak{f}}_{\ell}}\bigg)\Bigg\}\Bigg\{\sum_{\mathfrak{f}\in\mathfrak{F}}|\mathfrak{a}_{{\scriptscriptstyle\mathfrak{f}}}|^{2}\Bigg\}^{-j/2}\,.
\end{equation}
Hence, condition \textbf{(4\textnormal{\textbf{b}})} is satisfied.
Since $\mathscr{L}_{j}(\omega)\leq(M/m)^{j}$, it follows that condition \textbf{(4\textnormal{\textbf{c}})} is satisfied. We conclude that $\omega_{{\scriptscriptstyle\times}}\in\mathbf{\Omega}$.\\
We now prove item \textnormal{\textbf{(ii)}}. Suppose that $\lambda_{{\scriptscriptstyle1}},\ldots,\lambda_{{\scriptscriptstyle n}}$ are linearly independent over $\mathbb{Z}$. Since $\mathfrak{F}$ consists of integer linear combination of the frequencies $\lambda_{{\scriptscriptstyle\ell}}$, it follows from the definition of the coefficients $\mathfrak{a}_{{\scriptscriptstyle\mathfrak{f}}}$ that
\begin{equation}\label{eq:3.26}
\begin{split}
\underset{\sum_{\ell=1}^{j}\mathfrak{f}_{{\scriptscriptstyle\ell}}=0}{\sum_{\mathfrak{f}_{{\scriptscriptstyle1}},\ldots,\mathfrak{f}_{{\scriptscriptstyle j}}\in\mathfrak{F}}}\bigg(\prod_{\ell=1}^{j}\mathfrak{a}_{{\scriptscriptstyle\mathfrak{f}}_{\ell}}\bigg)&=\prod_{\ell=1}^{n}\Bigg\{\,\,\underset{\sum^{j}_{k=1}m_{{\scriptscriptstyle k}}=0}{\sum_{|m_{{\scriptscriptstyle1}}|\leq d_{{\scriptscriptstyle\ell}},\ldots,|m_{{\scriptscriptstyle j}}|\leq d_{{\scriptscriptstyle\ell}}}}\prod_{k=1}^{j}a_{m_{{\scriptscriptstyle k},\ell}}\Bigg\}\\
&=\prod_{\ell=1}^{n}\int\limits_{0}^{1}\Bigg(\sum^{d_{{\scriptscriptstyle\ell}}}_{m=-d_{{\scriptscriptstyle\ell}}}a_{{\scriptscriptstyle m,\ell}}\exp{(2\pi im_{{\scriptscriptstyle}}t)}\Bigg)^{j}\textit{d}t\\
&=\prod_{\ell=1}^{n}\int\limits_{0}^{1}\varphi^{j}_{{\scriptscriptstyle\ell}}(t)\textit{d}t=\int\limits_{\mathbf{t}\in[0,1)^{n}}\Phi^{j}_{{\scriptscriptstyle\times}}(\mathbf{t})\textit{d}\mathbf{t}\,.
\end{split}
\end{equation}
From \eqref{eq:3.25} and \eqref{eq:3.26} we obtain \eqref{eq:3.13}. This concludes the proof of Lemma $7$.
\end{proof}
\noindent
\textbf{Acknowledgements.} I would like to thank Yuk-Kam Lau for several fruitful conversations.
%\section{Plots and figures}

\end{document}